\documentclass[12pt]{amsart}

\usepackage{amsmath,amsthm,amsfonts,amssymb,anysize,bbm,euscript,stmaryrd,xcolor,mathscinet,mathtools,combelow}
\usepackage{mathrsfs}
\usepackage{tikz}
\usetikzlibrary{arrows}
\usepackage{graphicx}
\usepackage{subcaption}
\usepackage[normalem]{ulem}
\usepackage{tikz-cd}

\usepackage[hmargin=3cm,vmargin=3cm]{geometry}

\usepackage[pdftex,plainpages=false,hypertexnames=false,pdfpagelabels,breaklinks]{hyperref}
\hypersetup{
  colorlinks, linkcolor={blue},
  citecolor={medium-blue}, urlcolor={medium-blue}
}

\usepackage{palatino, euscript}

\usepackage{graphicx}
\usepackage{ifthen}
\usepackage[all]{xy}

\definecolor{dark-red}{rgb}{0.7,0.25,0.25}
\definecolor{dark-blue}{rgb}{0.15,0.15,0.55}
\definecolor{medium-blue}{rgb}{0,0,0.65}
\definecolor{DarkGreen}{RGB}{0,150,0}

\usepackage[backend=biber,style=alphabetic,doi=false,isbn=false,url=false,minnames=6,maxnames=6]{biblatex}
\renewbibmacro{in:}{%
 \ifentrytype{article}{}{\printtext{\bibstring{in}\intitlepunct}}}
\renewbibmacro*{volume+number+eid}{%
 \printtext{vol.}
 \printfield{volume}%
 \setunit*{\addnbspace}
 \printfield{number}%
 \setunit{\addcomma\space}%
 \printfield{eid}}
\DeclareFieldFormat[article]{number}{\mkbibparens{#1}}
\newcommand{\doi}[1]{\href{https://dx.doi.org/#1}{{\small DOI:#1}}}

\newcommand{\googlebooks}[1]{(preview at \href{https://books.google.com/books?id=#1}{google books})}

\newcommand{\numdam}[1]{}

\theoremstyle{plain}
\newtheorem{theorem}{Theorem}[section]
\newtheorem{thm-nono}{Theorem}
\newtheorem{proposition}[theorem]{Proposition}

\newtheorem{lemma}[theorem]{Lemma}

\theoremstyle{definition}
\newtheorem{definition}[theorem]{Definition}
\newtheorem{remark}[theorem]{Remark}
\newtheorem{example}[theorem]{Example}

\newcommand{\cat}[1]{\mathchoice
  {\ensuremath{\mbox{\bfseries {\upshape {#1}}}}}
  {\ensuremath{\mbox{\bfseries {\upshape {#1}}}}}
  {\scalebox{.7}{\ensuremath{\mbox{\bfseries {\upshape {#1}}}}}}
  {\scalebox{.5}{\ensuremath{\mbox{\bfseries {\upshape {#1}}}}}}%
  }

\def\mf{\mathfrak}
\newcommand{\foam}[1][]{#1\cat{Foam}}
\newcommand{\Cob}{\cat{Cob}}

\newcommand{\slnn}[1]{\mf{sl}_{#1}}

\newcommand{\glnn}[1]{\mf{gl}_{#1}}

\def\N{{\mathbbm N}}
\def\R{{\mathbbm R}}
\def\Z{{\mathbbm Z}}

\newcommand{\bracor}[1]{\left\llbracket #1 \right\rrbracket_\mathrm{or}}
\newcommand{\CKhor}{\mathrm{CKh}_\mathrm{or}}

\newcommand{\Hom}{\operatorname{Hom}}
\newcommand{\HOM}{\operatorname{HOM}}
\newcommand{\END}{\operatorname{END}}

\newcommand{\Kh}{\operatorname{Kh}}

\newcommand{\Khor}{\operatorname{Kh}_\mathrm{or}}

\newcommand{\oone}{\mathbf{1}}
\newcommand{\ox}{\mathbf{X}}

\usepackage{tikz}
\usetikzlibrary{cd}
\usetikzlibrary{calc}
\usetikzlibrary{decorations.markings}
\usetikzlibrary{decorations.pathreplacing}
\usetikzlibrary{arrows,shapes,positioning}
\tikzstyle directed=[postaction={decorate,decoration={markings,
    mark=at position #1 with {\arrow{>}}}}]
\tikzstyle rdirected=[postaction={decorate,decoration={markings,
    mark=at position #1 with {\arrow{<}}}}]
\tikzset{anchorbase/.style={baseline={([yshift=-0.5ex]current bounding box.center)}},
  tinynodes/.style={font=\tiny,text height=0.75ex,text depth=0.15ex},
  smallnodes/.style={font=\scriptsize,text height=0.75ex,text depth=0.15ex},
   >={Latex[length=1mm, width=1.5mm]}
}

\usetikzlibrary{matrix, arrows, calc}
\tikzset{frontline/.style={preaction={draw=white,-,line width=6pt}},}  

\newcommand{\ru}{to [out=0,in=270]}
\newcommand{\rd}{to [out=0,in=90]}
\newcommand{\ur}{to [out=90,in=180]}
\newcommand{\ul}{to [out=90,in=0]}
\newcommand{\lu}{to [out=180,in=270]}

\newcommand{\dl}{to [out=270,in=0]}
\newcommand{\pu}{to [out=90,in=270]}

\title{$\glnn{2}$ foams and the Khovanov homotopy type}
\author{Vyacheslav Krushkal}
\address{V.K.: Department of Mathematics, University of Virginia, Charlottesville VA 22904-4137}
\email{{krushkal@virginia.edu}}
\author{Paul Wedrich}

\address{P.W.: Max Planck Institute for Mathematics,
Vivatsgasse 7, 53111 Bonn, Germany} 
\address{P.W.: Mathematical Institute, University of Bonn,
Endenicher Allee 60, 53115 Bonn, Germany
\href{http://paul.wedrich.at}{paul.wedrich.at}}
\email{p.wedrich@gmail.com}

\begin{document}

\begin{abstract} The Blanchet link homology theory is an oriented model of Khovanov homology, functorial over the integers with respect to link cobordisms.
We formulate a stable homotopy refinement of the Blanchet theory, based on a comparison of the Blanchet and Khovanov chain complexes associated to link diagrams. The construction of the stable homotopy type relies on the signed Burnside category approach of Sarkar-Scaduto-Stoffregen. 
\end{abstract}

\maketitle

\section{Introduction}
The Khovanov homology \cite{MR1740682} is an invariant of oriented links in $\R^3$. Given such a link $L$, the associated invariant ${\rm Kh}(L)$ is a bigraded abelian group. A link cobordism $\Sigma$ in $\R^3\times [0,1]$ induces a map on Khovanov homology which is well-defined up to a sign, see Jacobsson \cite{MR2113903}. 
The map is defined by decomposing $\Sigma$ into elementary cobordisms, corresponding to Morse surgeries and Reidemeister moves on link diagrams; decompositions of isotopic link cobordisms are related by the {\em movie moves} of Carter--Saito~\cite{MR1238875}, and some of them indeed change the sign of the induced map \cite[Lemma 5.2]{MR2113903}.

Several models have been formulated (Blanchet \cite{MR2647055}, Caprau \cite{MR2443094}, Clark-Morrison-Walker \cite{MR2496052}, Sano \cite{2008.02131}, Vogel \cite{MR4096813}) to fix the sign ambiguity for cobordism maps in Khovanov homology, while producing isomorphic link homology groups. The resulting full functoriality is important in particular for the $4$-dimensional aspect of the Khovanov theory \cite{2019arXiv190712194M} and topological applications.

An entirely different refinement of Khovanov homology is due to Lipshitz and Sarkar. In \cite{MR3230817} they associate a suspension spectrum ${\mathcal X}(L)$ to each link diagram $L$, whose stable  homotopy type is an invariant of links in $\R^3$ and whose cohomology is the Khovanov homology of $L$.  It is shown in \cite{MR3189434} that a link cobordism $\Sigma: L_0\longrightarrow L_1$ in $\R^3\times [0,1]$ represented 
as a sequence of Reidemeister moves and elementary Morse cobordisms gives rise to a map of spectra ${\mathcal X}(L_1)\longrightarrow {\mathcal X}(L_0)$, whose induced map on cohomology is the Khovanov map $\Kh(\Sigma)\colon \Kh(L_0)\longrightarrow {\rm Kh}(L_1)$. The map of spectra associated to link cobordisms is conjectured \cite{MR3189434} to be well-defined, up to an overall sign -- like the map on Khovanov homology. 

In this paper we combine the two theories by constructing a stable homotopy refinement ${\mathcal X}_{\rm or}(L)$ of the Blanchet theory \cite{MR2647055}.

\begin{theorem} \label{main theorem:stable homotopy type} 
Let $L$ be an oriented link diagram. The stable homotopy type ${\mathcal X}_{\rm or}(L)$  is an invariant of the isotopy class of the link. Its cohomology is isomorphic to Blanchet's oriented model of Khovanov homology of $L$.
\end{theorem}

Our first motivation to pursue this result is the goal of developing a fully functorial stable homotopy invariant for links in $\R^3$ and their cobordisms. To this end, our construction fixes the sign ambiguity, like the Blanchet theory does for Khovanov homology. Building on  \cite{MR3189434}, we show (Lemma \ref{cob lemma}) that a link cobordism $L_0\longrightarrow L_1$, presented as a sequence of elementary cobordisms, gives rise to a map of spectra ${\mathcal X}_{\rm or}(L_1)\longrightarrow {\mathcal X}_{\rm or}(L_0)$. The induced map on cohomology is the map given by the Blanchet theory. Our methods do not address the conjecture that the map induced by a link cobordism $\Sigma$ is well-defined with respect to isotopies of $\Sigma$.

A strategy for proving Theorem~\ref{main theorem:stable homotopy type} should be well-known to the (hitherto possibly empty) intersection of two communities of experts. It relies on two technical tools:
\begin{itemize}
    \item a choice of a natural (though not necessarily canonical) isomorphism between the Khovanov and Blanchet theories, e.g. as elaborated by Ehrig--Stroppel--Tubbenhauer and Beliakova--Hogancamp--Putyra--Wehrli \cite{2016arXiv160108010E,2019arXiv190312194B}, and 
    \item the signed Burnside category framework of Sarkar--Scaduto--Stoffregen from \cite{MR4078823}, which accommodates such natural isomorphism,
\end{itemize}
to lift the reformulation of the Khovanov homotopy in terms of the Burnside category from Lawson--Lipshitz--Sarkar~\cite{MR4153651, MR3611723} to Blanchet's setting. The exposition in this paper follows the strategy outlined above, starting from the perspective of the Blanchet theory. In particular, we do not assume the reader to be familiar with the technical tools mentioned above.

Our emphasis on giving explicit constructions using the combinatorics of $\glnn{2}$ webs and foams is related to our second motivation, namely the problem of extending the Lipshitz--Sarkar construction to $\mathfrak{sl}_N$ (or, more accurately, $\glnn{N}$) homology theories for $N\geq 2$. The framed flow category construction of the Khovanov homotopy type in \cite{MR3230817}, as well as its reformulation in terms of the Burnside category \cite{MR4153651, MR3611723} utilize the existence of canonical generators of the Khovanov chain complex associated to a link diagram. Moreover, the components of the Khovanov differential have coefficients $ 0, 1$ with respect to these generators. As such, the differential may be combinatorially encoded using correspondences in the context of the Burnside category. 

The $\glnn{N}$ link homology theories are more intricate; in particular both the problem of identifying a basis and computing the differential are substantially more involved. The Blanchet theory, which fits in the foam category formulation\footnote{See \cite{1702.04140,MR3877770}, building on \cite{MR2491657,MR3426687,MR3545951}.} of $\glnn{N}$ link homology theories at $N=2$, is thus an interesting test case. Given an oriented link diagram, we use an additional combinatorial piece of data associated to the link diagram, a {\em flow} with values in $\{+1,-1\}$, to construct a canonical set of $\glnn{2}$ foam generators (see Section \ref{sec: Bases for web spaces}). The differential has coefficients in the set $\{ 0, \pm 1\}$ with respect to these generators. Importantly, Proposition \ref{prop:zipunzip} shows that the edge differential and the cobordism maps are {\em sign-coherent}, in the sense that for any two webs $V, W$ related by a zip/unzip or saddle foam, the associated map has coefficients in either $\{ 0, 1\}$ or $\{ 0, - 1\}$ with respect to the chosen bases. These results are then used in Section \ref{stable homotopy sec} to formulate the stable homotopy type ${\mathcal X}_{\rm or}(L)$ and to define the maps of spectra induced by elementary link cobordisms. The construction is given using the signed Burnside category framework of \cite{MR4078823}. In Section \ref{sec: Bl Khov} we show that ${\mathcal X}_{\rm or}(L)$ is homotopy equivalent to ${\mathcal X}(L)$ of \cite{MR3230817}.

Our construction of ${\mathcal X}_{\rm or}(L)$ can be seen as a step towards a construction of stable homotopy refinements of the $\glnn{N}$ link homology theories using framed flow category methods. However, there are several features that make the analysis substantially easier in the $N=2$ theory considered here, and new ideas are needed to address the $N>2$ case. See the work of Jones--Lobb-Sch\"utz \cite{MR4015262} in this direction. It would also be interesting to compare the present work with Kitchloo's constructions in \cite{2019arXiv191007443K}.   

Constructing a $\glnn{2}$ version of the tangle invariant of \cite{1706.023461} is outside the scope of this paper, but we expect that the methods considered here should be helpful for formulating it as well. 
We also hope that our results will contribute to the development of a stable homotopy refinement of the 4-manifold invariants from \cite{2019arXiv190712194M}; indeed this was a main motivation for our work.

{\bf Acknowledgements.} We would like to thank Robert Lipshitz and Sucharit Sarkar for helpful discussions related to this work.
The first named author also would like to thank Rostislav Akhmechet and Michael Willis for many conversations on the subject of stable homotopy refinement of link homology theories.

VK was supported in part by the Miller Institute for Basic Research in Science at UC Berkeley and Simons Foundation fellowship 608604. PW was
supported by the National Science Foundation under Grant No. DMS-1440140, while
in residence at the Mathematical Sciences Research Institute in Berkeley,
California, during the Spring 2020 semester. 

\section{Blanchet's oriented model for Khovanov homology} \label{Blanchet sec}
Here we recall the construction of Blanchet's oriented model of Khovanov homology using $\glnn{2}$ foams. We shall assume that the reader is familiar with Bar-Natan's description of Khovanov homology using \emph{dotted cobordisms} from \cite{MR2174270}. 

\begin{definition}
The dotted cobordism category $\Cob$ is the additive completion of the graded pre-additive category with
\begin{itemize}
    \item objects given by closed unoriented 1-dimensional manifolds properly embedded in $\R^2$ (and grading shifts thereof, encoded by powers of $q$),
    \item morphisms given by $\Z$-linear combinations of dotted cobordisms, i.e. unoriented 2-dimensional manifolds properly embedded in $\R^2\times I$, whose components may be decorated with \emph{dots}, considered up to isotopy rel boundary and the local relations
\begin{equation}\label{sl2closedfoam}
\begin{tikzpicture} [fill opacity=0.2, scale=.5,anchorbase]
	\filldraw [fill=blue] (0,0) circle (1);
	\draw (-1,0) .. controls (-1,-.4) and (1,-.4) .. (1,0);
	\draw[dashed] (-1,0) .. controls (-1,.4) and (1,.4) .. (1,0);
\end{tikzpicture}
\; = \; 0\; , \quad
\begin{tikzpicture} [fill opacity=0.2, scale=.5,anchorbase]
	\filldraw [fill=blue] (0,0) circle (1);
	\draw (-1,0) .. controls (-1,-.4) and (1,-.4) .. (1,0);
	\draw[dashed] (-1,0) .. controls (-1,.4) and (1,.4) .. (1,0);
	\node [opacity=1]  at (0,0.6) {$\bullet$};
\end{tikzpicture}
\; = \; 1
\; , \quad
\begin{tikzpicture} [fill opacity=0.2, scale=.4,anchorbase]
	\draw [fill=blue] (0,3.5) ellipse (1 and 0.5);
	\path [fill=blue] (0,0) ellipse (1 and 0.5);
	\draw (1,0) .. controls (1,-.66) and (-1,-.66) .. (-1,0);
	\draw[dashed] (1,0) .. controls (1,.66) and (-1,.66) .. (-1,0);
	\draw (1,3.5) to (1,0);
	\draw (-1,3.5) to (-1,0);
	\path[fill=blue, opacity=.3] (-1,3.5) .. controls (-1,2.84) and (1,2.84) .. (1,3.5) to
		(1,0) .. controls (1,.66) and (-1,.66) .. (-1,0) to cycle;
\end{tikzpicture}
\; = \;
\begin{tikzpicture} [fill opacity=0.2, scale=.4,anchorbase]
	\draw [fill=blue] (0,3.5) ellipse (1 and 0.5);
	\draw (-1,3.5) .. controls (-1,1.5) and (1,1.5) .. (1,3.5);
	\path [fill=blue, opacity=.3] (1,3.5) .. controls (1,2.84) and (-1,2.84) .. (-1,3.5) to
		(-1,3.5) .. controls (-1,1.5) and (1,1.5) .. (1,3.5);
	\path [fill=blue] (0,0) ellipse (1 and 0.5);
	\draw (1,0) .. controls (1,-.66) and (-1,-.66) .. (-1,0);
	\draw[dashed] (1,0) .. controls (1,.66) and (-1,.66) .. (-1,0);
	\draw (-1,0) .. controls (-1,2) and (1,2) .. (1,0);
	\path [fill=blue, opacity=.3] (1,0) .. controls (1,.66) and (-1,.66) .. (-1,0) to
		(-1,0) .. controls (-1,2) and (1,2) .. (1,0);
		\node[opacity=1] at (0,1) {$\bullet$};
\end{tikzpicture}
\; + \;
\begin{tikzpicture} [fill opacity=0.2, scale=.4,anchorbase]
	\draw [fill=blue] (0,3.5) ellipse (1 and 0.5);
	\draw (-1,3.5) .. controls (-1,1.5) and (1,1.5) .. (1,3.5);
	\path [fill=blue, opacity=.3] (1,3.5) .. controls (1,2.84) and (-1,2.84) .. (-1,3.5) to
		(-1,3.5) .. controls (-1,1.5) and (1,1.5) .. (1,3.5);
	\node[opacity=1] at (0,2.5) {$\bullet$};
	\path [fill=blue] (0,0) ellipse (1 and 0.5);
	\draw (1,0) .. controls (1,-.66) and (-1,-.66) .. (-1,0);
	\draw[dashed] (1,0) .. controls (1,.66) and (-1,.66) .. (-1,0);
	\draw (-1,0) .. controls (-1,2) and (1,2) .. (1,0);
	\path [fill=blue, opacity=.3] (1,0) .. controls (1,.66) and (-1,.66) .. (-1,0) to
		(-1,0) .. controls (-1,2) and (1,2) .. (1,0);
\end{tikzpicture}
\;,\quad 
\begin{tikzpicture}[fill opacity=.2, scale=.8, anchorbase]
\filldraw [fill=blue] (-1,-1) rectangle (1,1);
\node [opacity=1] at (0,-.15) {$\bullet$};
\node [opacity=1] at (0,.15) {$\bullet$};
\end{tikzpicture}
\;=\; 0.
\end{equation}
\item composition is given by (the bilinear extension of) stacking cobordisms,
\item the grading requires a dotted cobordism $C\colon q^k M_1\longrightarrow q^l M_2$ to satisfy 
\[2\#\textrm{dots}-\chi(C)=l-k.\]
\end{itemize}
\end{definition}
This cobordism category\footnote{Or, rather, its extension to a certain type of 3-category termed ``canopolis'' by Bar-Natan} categorifies the Temperley--Lieb skein theory of linear combinations of isotopy classes of unoriented planar curves, modulo the relation \[\bigcirc=(q+q^{-1})\,\emptyset,\] which models the pivotal tensor category of $U_q(\slnn{2})$-modules generated by the natural 2-dimensional module. 
\smallskip

Next we describe the category of Blanchet foams as an oriented extension of $\Cob$. Blanchet foams categorify the skein theory of $\glnn{2}$ webs, which models the pivotal tensor category of $U_q(\glnn{2})$-modules generated by the natural 2-dimensional module and its exterior square, see \cite{MR3263166,MR3912946,MR3709658}. 

Here we define a $\glnn{2}$ web to be an oriented, trivalent graph properly embedded in $\R^2$, with a flow on the edges taking values in the set $\{1,2\}$. Practically, this means a web is a graph drawn in the plane, with oriented edges labelled either $1$ or $2$ (and drawn solid or doubled respectively) and at vertices two $1$-labelled edges merge into a $2$-labelled one, or a $2$-labelled edge splits into two $1$-labelled ones. (Examples of (local) webs are shown in Lemma~\ref{lem:webisos} and Lemma~\ref{lem:regionsimpl}.)

We also need the notion of (embedded) $\glnn{2}$ foams, which form the morphisms in Blanchet's category.  
Foams for $\glnn{2}$ are ``trivalent surfaces'' which we describe next, properly embedded in $\R^2\times I$, equipped with a flow taking values in $\{1,2\}$. Foams are assembled from 
compact oriented $1$- or $2$-labelled surfaces, called \textit{facets}, which are glued along (part of) their boundary so that precisely two boundary components of
$1$-labeled facets are identified with a single boundary component of a
$2$-labeled facet with the opposite orientation. Points on such \textit{seams} have a neighbourhood modelled on the letter {\it Y} times an interval, and all other points have disk neighbourhoods\footnote{Foams for $\glnn{N}$ for $N\geq 3$ require other singularities, which can be avoided here.}. Facets with label $1$ are furthermore allowed to carry dots.

Given a $\glnn{2}$ foam $F$ or web $W$, we denote by $c(F)$ the unoriented surface obtained by deleting all 2-labelled facets in $F$, and similarly $c(W)$ denotes the result of deleting all $2$-labelled edges in $W$. 

\begin{definition} \label{def:Tfoam}
The  $\glnn{2}$ foam category $\foam$  is the additive completion of the graded pre-additive category with:
\begin{itemize}
\item objects given by webs embedded in $\R^2$ (and grading shifts thereof),
\item morphisms given by $\Z$-linear combinations of $\glnn{2}$ foams
properly embedded in $\R^2\times [0,1]$, considered up to isotopy rel boundary and Blanchet's local relations \cite{MR2647055}, which include \eqref{sl2closedfoam} on 1-labelled facets as well as:

\begin{equation} \label{sl2neckcutting_enh_2lab}
\begin{tikzpicture} [fill opacity=0.3, scale=.5,anchorbase]
	\draw [fill=yellow , fill opacity=0.3] (0,0) circle (1);
	\draw (-1,0) .. controls (-1,-.4) and (1,-.4) .. (1,0);
	\draw[dashed] (-1,0) .. controls (-1,.4) and (1,.4) .. (1,0);
\end{tikzpicture}
\;=\; -1
\;, \qquad
\begin{tikzpicture} [scale=0.4,fill opacity=0.2,  decoration={markings, mark=at position 0.6 with {\arrow{>}};  },anchorbase]
	\draw [fill=yellow , fill opacity=0.3] (0,3.5) ellipse (1 and 0.5);
	\path [fill=yellow , fill opacity=0.3] (0,0) ellipse (1 and 0.5);
	\draw (1,0) .. controls (1,-.66) and (-1,-.66) .. (-1,0);
	\draw[dashed] (1,0) .. controls (1,.66) and (-1,.66) .. (-1,0);
	\draw (1,3.5) -- (1,0);
	\draw (-1,3.5) -- (-1,0);
	\path[fill=yellow, opacity=.3] (-1,3.5) .. controls (-1,2.84) and (1,2.84) .. (1,3.5) --
		(1,0) .. controls (1,.66) and (-1,.66) .. (-1,0) -- cycle;
\end{tikzpicture}
\; = \; - \;
\begin{tikzpicture} [scale=0.4,fill opacity=0.2,  decoration={markings, mark=at position 0.6 with {\arrow{>}};  },anchorbase]
	\draw [fill=yellow , fill opacity=0.3] (0,3.5) ellipse (1 and 0.5);
	\draw (-1,3.5) .. controls (-1,1.5) and (1,1.5) .. (1,3.5);
	\path [fill=yellow, opacity=.3] (1,3.5) .. controls (1,2.84) and (-1,2.84) .. (-1,3.5) --
		(-1,3.5) .. controls (-1,1.5) and (1,1.5) .. (1,3.5);
	\path [fill=yellow , fill opacity=0.3] (0,0) ellipse (1 and 0.5);
	\draw (1,0) .. controls (1,-.66) and (-1,-.66) .. (-1,0);
	\draw[dashed] (1,0) .. controls (1,.66) and (-1,.66) .. (-1,0);
	\draw (-1,0) .. controls (-1,2) and (1,2) .. (1,0);
	\path [fill=yellow, opacity=.3] (1,0) .. controls (1,.66) and (-1,.66) .. (-1,0) --
		(-1,0) .. controls (-1,2) and (1,2) .. (1,0);
\end{tikzpicture}
\;,\qquad
\begin{tikzpicture} [fill opacity=0.2,decoration={markings, mark=at position 0.6 with {\arrow{>}}; }, scale=.65,anchorbase]
	\filldraw [fill=blue] (0,0) circle (1);
	\path [fill=yellow , fill opacity=0.3] (0,0) ellipse (1 and 0.3); 	
	\draw [very thick, blue, postaction={decorate}] (-1,0) .. controls (-1,-.4) and (1,-.4) .. (1,0);
	\draw[very thick, blue, dashed] (-1,0) .. controls (-1,.4) and (1,.4) .. (1,0);
	\node [opacity=1]  at (0,0.7) {$\bullet$};
	\node [opacity=1]  at (0,-0.7) {$\bullet$};
	\node [opacity=1] at (-0.3,0.7) {\tiny $\alpha$};
	\node [opacity=1] at (-0.3,-0.7) {\tiny $\beta$};
\end{tikzpicture}
\quad = \delta_{\alpha,1}\delta_{\beta,0} - \delta_{\alpha,0}\delta_{\beta,1}
\end{equation}


\begin{equation} \label{sl2Fig5Blanchet_1}
\begin{tikzpicture} [scale=.5,fill opacity=0.2,decoration={markings, mark=at position 0.5 with {\arrow{>}}; },anchorbase]
	\draw[very thick, postaction={decorate}] (2,0) -- (.75,0);
	\draw[very thick, postaction={decorate}] (-.75,0) -- (-2,0);
	\draw[very thick, postaction={decorate}] (-.75,0) .. controls (-.5,-.5) and (.5,-.5) .. (.75,0);
	\draw[dashed, double, postaction={decorate}] (.75,0) .. controls (.5,.5) and (-.5,.5) .. (-.75,0);
	\draw (-2,0) -- (-2,3.5);
	\draw (2,0) -- (2,3.5);
	\path [fill=blue] (-2,3.5) -- (-.75,3.5) -- (-.75,0) -- (-2,0) -- cycle;
	\path [fill=blue] (2,3.5) -- (.75,3.5) -- (.75,0) -- (2,0) -- cycle;
	\path [fill=yellow , fill opacity=0.3] (-.75,3.5) .. controls (-.5,4) and (.5,4) .. (.75,3.5) --
		(.75,0) .. controls (.5,.5) and (-.5,.5) .. (-.75,0);
	\path [fill=blue] (-.75,3.5) .. controls (-.5,3) and (.5,3) .. (.75,3.5) --
		(.75,0) .. controls (.5,-.5) and (-.5,-.5) .. (-.75,0);
	\draw [very thick, blue, postaction={decorate}] (-.75,0) -- (-.75,3.5);
	\draw [very thick, blue, postaction={decorate}] (.75, 3.5) -- (.75,0);
	\draw[very thick, postaction={decorate}] (2,3.5) -- (.75,3.5);
	\draw[very thick, postaction={decorate}] (-.75,3.5) -- (-2,3.5);
	\draw[very thick, postaction={decorate}] (-.75,3.5) .. controls (-.5,3) and (.5,3) .. (.75,3.5);
	\draw[double, postaction={decorate}] (.75,3.5) .. controls (.5,4) and (-.5,4) .. (-.75,3.5);
\end{tikzpicture}
\; = \;
\begin{tikzpicture} [scale=.5,fill opacity=0.2,decoration={markings, mark=at position 0.5 with {\arrow{>}}; },anchorbase]
	\draw[very thick, postaction={decorate}] (2,0) -- (.75,0);
	\draw[very thick, postaction={decorate}] (-.75,0) -- (-2,0);
	\draw[very thick, postaction={decorate}] (-.75,0) .. controls (-.5,-.5) and (.5,-.5) .. (.75,0);
	\draw[dashed, double, postaction={decorate}] (.75,0) .. controls (.5,.5) and (-.5,.5) .. (-.75,0);
	\draw (-2,0) -- (-2,3.5);
	\draw (2,0) -- (2,3.5);
	\path [fill=blue] (-2,3.5) -- (-.75,3.5) -- (-.75,0) -- (-2,0) -- cycle;
	\path [fill=blue] (2,3.5) -- (.75,3.5) -- (.75,0) -- (2,0) -- cycle;
	\path [fill=blue] (-.75,3.5) .. controls (-.75,1.5) and (.75,1.5) .. (.75,3.5) --
			(.75, 0) .. controls (.75,2) and (-.75,2) .. (-.75,0);
	\path [fill=yellow , fill opacity=0.3] (-.75,3.5) .. controls (-.5,4) and (.5,4) .. (.75,3.5) --
			(.75,3.5) .. controls (.75,1.5) and (-.75,1.5) .. (-.75,3.5);
	\path [fill=blue] (-.75,3.5) .. controls (-.5,3) and (.5,3) .. (.75,3.5) --
			(.75,3.5) .. controls (.75,1.5) and (-.75,1.5) .. (-.75,3.5);
	\path [fill=blue] (-.75, 0) .. controls (-.75,2) and (.75,2) .. (.75,0) --
			(.75,0) .. controls (.5,-.5) and (-.5,-.5) .. (-.75,0);
	\path [fill=yellow , fill opacity=0.3] (-.75, 0) .. controls (-.75,2) and (.75,2) .. (.75,0) --
			(.75,0) .. controls (.5,.5) and (-.5,.5) .. (-.75,0);
	\draw [very thick, blue, postaction={decorate}] (.75,3.5) .. controls (.75,1.5) and (-.75,1.5) .. (-.75,3.5);
	\draw [very thick, blue, postaction={decorate}] (-.75, 0) .. controls (-.75,2) and (.75,2) .. (.75,0);
	\draw[very thick, postaction={decorate}] (2,3.5) -- (.75,3.5);
	\draw[very thick, postaction={decorate}] (-.75,3.5) -- (-2,3.5);
	\draw[very thick, postaction={decorate}] (-.75,3.5) .. controls (-.5,3) and (.5,3) .. (.75,3.5);
	\draw[double, postaction={decorate}] (.75,3.5) .. controls (.5,4) and (-.5,4) .. (-.75,3.5);
\end{tikzpicture}
\;, \quad
\begin{tikzpicture} [scale=.5,fill opacity=0.2,decoration={markings, mark=at position 0.5 with {\arrow{>}}; },anchorbase]
	\draw[double, postaction={decorate}] (.75,0) -- (2,0);
	\draw[double, postaction={decorate}] (-2,0) -- (-.75,0);
	\draw[very thick, postaction={decorate}] (-.75,0) .. controls (-.5,-.5) and (.5,-.5) .. (.75,0);
	\draw[very thick, dashed, postaction={decorate}] (-.75,0) .. controls (-.5,.5) and (.5,.5) .. (.75,0);
	\draw (-2,0) -- (-2,3.5);
	\draw (2,0) -- (2,3.5);
	\path [fill=yellow, opacity=.3] (-2,3.5) -- (-.75,3.5) -- (-.75,0) -- (-2,0) -- cycle;
	\path [fill=yellow, opacity=.3] (2,3.5) -- (.75,3.5) -- (.75,0) -- (2,0) -- cycle;
	\path [fill=blue] (-.75,3.5) .. controls (-.5,4) and (.5,4) .. (.75,3.5) --
		(.75,0) .. controls (.5,.5) and (-.5,.5) .. (-.75,0);
	\path [fill=blue] (-.75,3.5) .. controls (-.5,3) and (.5,3) .. (.75,3.5) --
		(.75,0) .. controls (.5,-.5) and (-.5,-.5) .. (-.75,0);
	\draw [very thick, blue, postaction={decorate}] (-.75,0) -- (-.75,3.5);
	\draw [very thick, blue, postaction={decorate}] (.75, 3.5) -- (.75,0);
	\draw[double, postaction={decorate}] (.75,3.5) -- (2,3.5);
	\draw[double, postaction={decorate}] (-2,3.5) -- (-.75,3.5);
	\draw[very thick, postaction={decorate}] (-.75,3.5) .. controls (-.5,3) and (.5,3) .. (.75,3.5);
	\draw[very thick, postaction={decorate}] (-.75,3.5) .. controls (-.5,4) and (.5,4) .. (.75,3.5);
\end{tikzpicture}
\; = \;
\begin{tikzpicture} [scale=.5,fill opacity=0.2,decoration={markings, mark=at position 0.5 with {\arrow{>}}; },anchorbase]
	\draw[double, postaction={decorate}] (.75,0) -- (2,0);
	\draw[double, postaction={decorate}] (-2,0) -- (-.75,0);
	\draw[very thick, postaction={decorate}] (-.75,0) .. controls (-.5,-.5) and (.5,-.5) .. (.75,0);
	\draw[very thick, dashed, postaction={decorate}] (-.75,0) .. controls (-.5,.5) and (.5,.5) .. (.75,0);
	\draw (-2,0) -- (-2,3.5);
	\draw (2,0) -- (2,3.5);
	\path [fill=yellow,opacity=.3] (-2,3.5) -- (-.75,3.5) -- (-.75,0) -- (-2,0) -- cycle;
	\path [fill=yellow, opacity=.3] (2,3.5) -- (.75,3.5) -- (.75,0) -- (2,0) -- cycle;
	\path [fill=yellow, opacity=.3] (-.75,3.5) .. controls (-.75,1.5) and (.75,1.5) .. (.75,3.5) --
			(.75, 0) .. controls (.75,2) and (-.75,2) .. (-.75,0);
	\path [fill=blue] (-.75,3.5) .. controls (-.5,4) and (.5,4) .. (.75,3.5) --
			(.75,3.5) .. controls (.75,1.5) and (-.75,1.5) .. (-.75,3.5);
	\path [fill=blue] (-.75,3.5) .. controls (-.5,3) and (.5,3) .. (.75,3.5) --
			(.75,3.5) .. controls (.75,1.5) and (-.75,1.5) .. (-.75,3.5);
	\path [fill=blue] (-.75, 0) .. controls (-.75,2) and (.75,2) .. (.75,0) --
			(.75,0) .. controls (.5,-.5) and (-.5,-.5) .. (-.75,0);
	\path [fill=blue] (-.75, 0) .. controls (-.75,2) and (.75,2) .. (.75,0) --
			(.75,0) .. controls (.5,.5) and (-.5,.5) .. (-.75,0);
	\draw [very thick, blue, postaction={decorate}] (.75,3.5) .. controls (.75,1.5) and (-.75,1.5) .. (-.75,3.5);
	\draw [very thick, blue, postaction={decorate}] (-.75, 0) .. controls (-.75,2) and (.75,2) .. (.75,0);
	\draw[double, postaction={decorate}] (.75,3.5) -- (2,3.5);
	\draw[double, postaction={decorate}] (-2,3.5) -- (-.75,3.5);
	\draw[very thick, postaction={decorate}] (-.75,3.5) .. controls (-.5,3) and (.5,3) .. (.75,3.5);
	\draw[very thick, postaction={decorate}] (-.75,3.5) .. controls (-.5,4) and (.5,4) .. (.75,3.5);
	\node [opacity=1] at (0,3.5) {$\bullet$};
\end{tikzpicture}
\; - \;
\begin{tikzpicture} [scale=.5,fill opacity=0.2,decoration={markings, mark=at position 0.5 with {\arrow{>}}; },anchorbase]
	\draw[double, postaction={decorate}] (.75,0) -- (2,0);
	\draw[double, postaction={decorate}] (-2,0) -- (-.75,0);
	\draw[very thick, postaction={decorate}] (-.75,0) .. controls (-.5,-.5) and (.5,-.5) .. (.75,0);
	\draw[very thick, dashed, postaction={decorate}] (-.75,0) .. controls (-.5,.5) and (.5,.5) .. (.75,0);
	\draw (-2,0) -- (-2,3.5);
	\draw (2,0) -- (2,3.5);
	\path [fill=yellow,opacity=.3] (-2,3.5) -- (-.75,3.5) -- (-.75,0) -- (-2,0) -- cycle;
	\path [fill=yellow, opacity=.3] (2,3.5) -- (.75,3.5) -- (.75,0) -- (2,0) -- cycle;
	\path [fill=yellow, opacity=.3] (-.75,3.5) .. controls (-.75,1.5) and (.75,1.5) .. (.75,3.5) --
			(.75, 0) .. controls (.75,2) and (-.75,2) .. (-.75,0);
	\path [fill=blue] (-.75,3.5) .. controls (-.5,4) and (.5,4) .. (.75,3.5) --
			(.75,3.5) .. controls (.75,1.5) and (-.75,1.5) .. (-.75,3.5);
	\path [fill=blue] (-.75,3.5) .. controls (-.5,3) and (.5,3) .. (.75,3.5) --
			(.75,3.5) .. controls (.75,1.5) and (-.75,1.5) .. (-.75,3.5);
	\path [fill=blue] (-.75, 0) .. controls (-.75,2) and (.75,2) .. (.75,0) --
			(.75,0) .. controls (.5,-.5) and (-.5,-.5) .. (-.75,0);
	\path [fill=blue] (-.75, 0) .. controls (-.75,2) and (.75,2) .. (.75,0) --
			(.75,0) .. controls (.5,.5) and (-.5,.5) .. (-.75,0);
	\draw [very thick, blue, postaction={decorate}] (.75,3.5) .. controls (.75,1.5) and (-.75,1.5) .. (-.75,3.5);
	\draw [very thick, blue, postaction={decorate}] (-.75, 0) .. controls (-.75,2) and (.75,2) .. (.75,0);
	\draw[double, postaction={decorate}] (.75,3.5) -- (2,3.5);
	\draw[double, postaction={decorate}] (-2,3.5) -- (-.75,3.5);
	\draw[very thick, postaction={decorate}] (-.75,3.5) .. controls (-.5,3) and (.5,3) .. (.75,3.5);
	\draw[very thick, postaction={decorate}] (-.75,3.5) .. controls (-.5,4) and (.5,4) .. (.75,3.5);
	\node [opacity=1] at (0,0) {$\bullet$};
\end{tikzpicture}
\end{equation}

Here we shade $1$-labelled facets blue and $2$-labelled facets yellow\footnote{In grayscale, blue appears darker than yellow.}. An essential feature of Blanchet's relations is that swapping all orientations introduces a minus sign in both of the relations in \eqref{sl2Fig5Blanchet_1}.

\item composition given by (the bilinear extension of) stacking foams,
\item the grading requires a foam $F\colon q^k W_1\longrightarrow q^l W_2$ to satisfy 
\[2\#\textrm{dots}-\chi(c(F))=k-l.\]
In particular, for any objects $W_1$, $W_2$, the morphism spaces $\Hom_{\foam}(W_1,W_2)$ are abelian groups. To capture morphisms of various degrees, we will also consider the graded abelian groups:
\[
\HOM_{\foam}(W_1,W_2):=\bigoplus_{k}\Hom_{\foam}(q^{-k}W_1,W_2)\]

\end{itemize}
\end{definition}

We record one important consequence of these relations, namely the dot-sliding
relation
\begin{equation}
\label{eq:dotslide}
\begin{tikzpicture} [scale=.8,fill opacity=0.2,  anchorbase, scale=.6]               
	\draw [very thick, blue] (0,-2) -- (0,0);
	\draw [double] (-1.5,-2) -- (0,-2);
	\draw [very thick] (0,-2) -- (1.8,-1);
	\draw [very thick] (0,-2) -- (1.5,-3);
	\filldraw [fill=blue] (0,0) -- (1.8,1) -- (1.8,-1) -- (0,-2) -- cycle;
	\path[fill=yellow, fill opacity=0.3] (-1.5,0) -- (0,0) -- (0,-2) -- (-1.5,-2) -- cycle;
	\draw (-1.5,0) -- (-1.5,-2);
	\filldraw [fill=blue] (0,0) -- (1.5,-1) -- (1.5,-3) -- (0,-2) -- cycle;
	\draw [double] (-1.5,0) -- (0,0);
	\draw [very thick] (0,0) -- (1.8,1);
	\draw [very thick] (0,0) -- (1.5,-1);
	\node [opacity=1] at (1,-0.3) {$\bullet$};
\end{tikzpicture} 
\; = \; - \;
\begin{tikzpicture} [scale=.6,fill opacity=0.2,  anchorbase, scale=.75]                                      	
	\draw [very thick, blue] (0,-2) -- (0,0);
	\draw [double] (-1.5,-2) -- (0,-2);
	\draw [very thick] (0,-2) -- (1.8,-1);
	\draw [very thick] (0,-2) -- (1.5,-3);
	\filldraw [fill=blue] (0,0) -- (1.8,1) -- (1.8,-1) -- (0,-2) -- cycle;
	\path[fill=yellow, fill opacity=0.3] (-1.5,0) -- (0,0) -- (0,-2) -- (-1.5,-2) -- cycle;
	\draw (-1.5,0) -- (-1.5,-2);
	\filldraw [fill=blue] (0,0) -- (1.5,-1) -- (1.5,-3) -- (0,-2) -- cycle;
	\draw [double] (-1.5,0) -- (0,0);
	\draw [very thick] (0,0) -- (1.8,1);
	\draw [very thick] (0,0) -- (1.5,-1);
	\node [opacity=1] at (1.1,-1.8) {$\bullet$};
\end{tikzpicture} .
\end{equation}

The presentation of the relation \eqref{sl2closedfoam} and \eqref{sl2Fig5Blanchet_1} is imported from
Lauda--Queffelec--Rose~\cite{MR3426687}, where it is shown that  Blanchet's foam
relations also arise in certain quotients of categorified quantum groups of
type A. For a generalisation to foams for $\glnn{N}$, see 
Robert--Wagner~\cite{1702.04140}.

\begin{remark}
\label{rem:def}
The last relation in \eqref{sl2closedfoam} (and thus also \eqref{eq:dotslide}) can be deformed to yield a foam-based construction
of Lee's deformed Khovanov homology \cite{MR2173845} or an equivariant link
homology~\cite{MR2232858}. For a careful study of such deformations, see \cite{2016arXiv160108010E,2020arXiv200414197K}. For the $\glnn{N}$ case see \cite{MR3590355}.
\end{remark}

\begin{remark} After specialising the ground ring from $\Z$ to $\mathbb{F}_2$, the map which deletes $2$-labelled facets and edges induces a full, essentially surjective functor
\[c\colon \foam_{\mathbb{F}_2} \longrightarrow \Cob_{\mathbb{F}_2},\]
as can be seen by comparing Blanchet's foam relations with Bar-Natan's cobordism relations. In fact, $c$ is also faithful. For refined versions of this statement see Queffelec--Wedrich~\cite{1806.03416} and  Beliakova--Hogancamp--Putyra--Wehrli~\cite{2019arXiv190312194B}.
\end{remark}
  
  \subsection{Blanchet's oriented model} \label{Blanchet review sec}
  To crossings in oriented link diagrams, Blanchet \cite{MR2647055} associates chain complexes of (local) webs and foams:
  \begin{equation*}
  \bracor{
  \begin{tikzpicture}[anchorbase,xscale=.45,yscale=.45]
    \draw [very thick, ->] (1,0) \pu (0,1.5);
    \draw [white, line width=.15cm] (0,0) \pu (1,1.5);
    \draw [very thick, ->] (0,0) \pu (1,1.5);
  \end{tikzpicture}
  }
  \!\!:=\!\left( 0 \longrightarrow 
  \uwave{ q^{-1}
\begin{tikzpicture}[anchorbase,scale=.45]
\draw [very thick, ->] (0,0) -- (0,1.5);
\draw [very thick, ->] (1,0) -- (1,1.5);
\end{tikzpicture}
}
\longrightarrow
q^{-2}
\begin{tikzpicture}[anchorbase,xscale=-.5,yscale=.5]
\draw [very thick] (0,0) \pu (0,0.1) \ur (.5,.6);
\draw [very thick] (1,0) \pu (1,0.1) \ul (.5,.6);
\draw [double] (.5,.6) \pu (.5,.9);
\draw [very thick,->] (.5,.9) \lu (0,1.3) \pu (0,1.5);
\draw [very thick,->] (.5,.9) \ru (1,1.3) \pu (1,1.5);
\end{tikzpicture}
\to 0
\right) 
  ,   \;\;
    \bracor{
  \begin{tikzpicture}[anchorbase,xscale=-.45,yscale=.45]
    \draw [very thick, ->] (1,0) \pu (0,1.5);
    \draw [white, line width=.15cm] (0,0) \pu (1,1.5);
    \draw [very thick, ->] (0,0) \pu (1,1.5);
  \end{tikzpicture}
  }
 \!\! :=\!\left( 0 \to 
       q^{2}
    \begin{tikzpicture}[anchorbase,scale=.45]
\draw [very thick] (0,0) \pu (0,0.1) \ur (.5,.6);
\draw [very thick] (1,0) \pu (1,0.1) \ul (.5,.6);
\draw [double] (.5,.6) \pu (.5,.9);
\draw [very thick,->] (.5,.9) \lu (0,1.3) \pu (0,1.5);
\draw [very thick,->] (.5,.9) \ru (1,1.3) \pu (1,1.5);
    \end{tikzpicture} 
     \to 
        \uwave{
  q^{1}
  \begin{tikzpicture}[anchorbase,scale=.45]
    \draw [very thick, ->] (0,0) -- (0,1.5);
    \draw [very thick, ->] (1,0) -- (1,1.5);
  \end{tikzpicture}
    }
 \to 0
    \right)
\end{equation*}  
with the \uwave{underlined} terms in homological degree zero and with differentials modelled by the so-called zip and unzip foams, cf. \cite[Section 3.2]{MR2647055}:
\[\begin{tikzpicture} [fill opacity=0.2,  decoration={markings, mark=at position 0.6 with {\arrow{>}};  }, scale=.4,anchorbase]
  \draw [very thick, blue] (.8,3.7) .. controls (.75,2.2) and (-.75,2.2) .. (-.8,4.3);
  \fill [yellow,opacity=0.5] (.8,3.7) .. controls (.75,2.2) and (-.75,2.2) .. (-.8,4.3);
  \fill[blue] (-.8,4.3) to [out=200,in=340] (-2.3,4.2) 
  to  (-2.3,2.2) to (1.1,.95) to (1.1,2.95) to [out=160,in=280] (.8,3.7) 
  .. controls (.75,2.2) and (-.75,2.2) .. (-.8,4.3);
  \fill[blue] (-.8,4.3) to [out=120,in=340] (-1.3,4.9) 
  to  (-1.3,2.9) to (2.1,1.65) to (2.1,3.65) to [out=160,in=20] (.8,3.7)
  .. controls (.75,2.2) and (-.75,2.2) .. (-.8,4.3);
  \draw [double, postaction={decorate}] (.8,3.7) to (-.8,4.3);
  \draw[very thick] (-.8,4.3) to [out=200,in=340] (-2.3,4.2) 
  to  (-2.3,2.2) to (1.1,.95) to (1.1,2.95) to [out=160,in=280] (.8,3.7);
  \draw[very thick] (-.8,4.3) to [out=120,in=340] (-1.3,4.9) 
  to  (-1.3,2.9) to (2.1,1.65) to (2.1,3.65) to [out=160,in=20] (.8,3.7);
\end{tikzpicture}
\;\;,
\quad
\begin{tikzpicture} [fill opacity=0.2,  decoration={markings, mark=at position 0.6 with {\arrow{>}};  }, scale=.4,anchorbase]
  \draw [very thick, blue] (.8,3.7) .. controls (.75,5.2) and (-.75,5.2) .. (-.8,4.3);
  \fill [yellow,opacity=0.5] (.8,3.7) .. controls (.75,5.2) and (-.75,5.2) .. (-.8,4.3);
  \fill[blue] (-.8,4.3) to [out=200,in=340] (-2.3,4.2) 
  to  (-2.3,6.2) to (1.1,4.95) to (1.1,2.95) to [out=160,in=280] (.8,3.7) 
  .. controls (.75,5.2) and (-.75,5.2) .. (-.8,4.3);
  \fill[blue]  (-.8,4.3) to [out=120,in=340] (-1.3,4.9) 
  to  (-1.3,6.9) to (2.1,5.65) to (2.1,3.65) to [out=160,in=20] (.8,3.7)
  .. controls (.75,5.2) and (-.75,5.2) .. (-.8,4.3);
  \draw [double, postaction={decorate}] (.8,3.7) to (-.8,4.3);
  \draw[very thick] (-.8,4.3) to [out=200,in=340] (-2.3,4.2) 
  to  (-2.3,6.2) to (1.1,4.95) to (1.1,2.95) to [out=160,in=280] (.8,3.7);
  \draw[very thick] (-.8,4.3) to [out=120,in=340] (-1.3,4.9) 
  to  (-1.3,6.9) to (2.1,5.65) to (2.1,3.65) to [out=160,in=20] (.8,3.7);
\end{tikzpicture} 
\]
For a link diagram $L$ (with an ordering of the crossings), the chain complex $\bracor{L}$ is defined as the total complex of the hyper\emph{cube of resolutions}, obtained by resolving each crossing locally as shown above. 

As in Bar-Natan's description of Khovanov homology, when considered as an object of the bounded homotopy category of chain complexes of webs and foams, denoted $K^b(\foam)$, the complex $\bracor{L}$ is an invariant of the link represented by $L$, well-defined up to isomorphism. In particular, Reidemeister moves induce chain homotopy equivalences between the associated complexes.

The passage from $\bracor{L}$ to a chain complex of graded abelian groups proceeds in analogy with Khovanov's original theory, using a \emph{TQFT for webs and foams} that can now be described as a representable functor.

The Blanchet--Khovanov chain complex of $L$ is the chain complex of graded abelian groups defined as
\begin{equation} \label{Hom eq}
\CKhor(L):= \HOM_{\foam}^\bullet(\emptyset,\bracor{L})
\end{equation}
  where $\HOM_{\foam}^\bullet(\emptyset,\bracor{L})$ denotes the chain complex of bihomogeneous maps between the $\Z\times\Z$-graded objects $\emptyset$ and $\bracor{L}$, with the differential induced by the differential on the target (the source has trivial differential).
  
The oriented (or $\glnn{2}$) Khovanov homology of $L$ is the bigraded abelian group 
\[\Khor(L):=H^\bullet(\CKhor(L)).\]

It follows from the discussion above that $\Khor(L)$ is an invariant of the link represented by $L$, well-defined up to isomorphism of bigraded abelian groups. In fact, these isomorphisms can be chosen coherently, as shown by Blanchet's main result, which we paraphrase as follows.

\begin{theorem} \label{Blanchet thm}
Let $C$ be an oriented smooth link cobordism, properly embedded in $\R^3\times I$, between oriented links represented by diagrams $L$ and $L'$. Suppose that $C$ is in generic position, so that it can be represented by a finite sequence of Reidemeister and Morse moves, transforming $L$ into $L'$. There exists an assignment of Reidemeister isomorphisms and Morse maps on the level of $\Khor$, so that the composite map $\Khor(C)$ is invariant under isotopy rel boundary on $C$. As a consequence, $\glnn{2}$ Khovanov homology constitutes a functor:
\[\Khor\colon \begin{Bmatrix}
\text{links embedded in } \R^3
\\
\text{link cobordisms in } \R^3\times I \text{/isotopy}
\end{Bmatrix} \xrightarrow{} \mathrm{gr}^{\Z\times \Z}\cat{Abgrp} \]
\end{theorem}

Other functorial versions of Khovanov homology have been constructed by Caprau~\cite{MR2443094}, Clark--Morrison--Walker~\cite{MR2496052}, Sano \cite{2008.02131}, and Vogel~\cite{MR4096813}. Blanchet's approach to functorial link homologies via foams works over the integers and is distinguished in that it extends to $\glnn{N}$, see Ehrig--Tubbenhauer--Wedrich~\cite{MR3877770}, and to links in $S^3$, see Morrison--Walker--Wedrich~\cite{2019arXiv190712194M}.

\section{Bases for web spaces} \label{sec: Bases for web spaces}

The goal of this section is to introduce generators of the webs spaces, and to analyze the coefficients of the Blanchet-Khovanov differential, as well as of the elementary link cobordisms with respect to these generators. These results underlie the formulation of the stable homotopy type in Section \ref{stable homotopy sec}.

\subsection{Construction of bases}

The following local relations hold in the category $\foam$ and will be used
throughout this section to simplify webs and foams.

\begin{lemma}\label{lem:webisos}
There are isomorphisms between webs in $\foam$ which differ only in a disk as
shown (or as shown, but with all orientations reversed):
\begin{gather}
\label{eq:circles}
\begin{tikzpicture}[fill opacity=.2,anchorbase,scale=.3]
\draw[very thick, directed=.55] (1,0) to [out=0,in=270] (2,1) to [out=90,in=0] (1,2)to [out=180,in=90] (0,1)to [out=270,in=180] (1,0);
\end{tikzpicture} 
\quad\cong\quad
q\;\emptyset \oplus q^{-1} \emptyset
\quad,\quad
\begin{tikzpicture}[fill opacity=.2,anchorbase,scale=.3]
\draw[double, directed=.55] (1,0) to [out=0,in=270] (2,1) to [out=90,in=0] (1,2)to [out=180,in=90] (0,1)to [out=270,in=180] (1,0);
\end{tikzpicture} 
\quad\cong\quad
\emptyset 
\\
\label{eq:squares}
\begin{tikzpicture}[anchorbase,scale=.5]
  \draw [double,->] (0,0) \ur  (.5,.66) (.5,1.33) \lu (0,2);
  \draw [very thick,<-] (1,0)to (1,0.2) \ul (.5,.66);
  \draw [very thick] (1,2)to (1,1.8) \dl (.5,1.33);
  \draw [very thick] (0.5,0.66) to (0.5,1.33);
  \end{tikzpicture}
  \quad =\quad
  \begin{tikzpicture}[anchorbase,scale=.5]
    \draw [double,->] (0,0) to  (0,2);
  \draw [very thick,->] (1,2) to (1,0);
  \end{tikzpicture}
\quad, \quad
\begin{tikzpicture}[anchorbase,scale=.5]
\draw [double,->] (0,0) to  (0,2);
\draw [double,->] (1,2) to (1,0);
\end{tikzpicture}
\quad \cong \quad
\begin{tikzpicture}[anchorbase,scale=.5]
\draw [double,->] (0,0) to (0,.5) to [out=90,in=90] (1,.5) to (1,0);
\draw [double,->] (1,2) to (1,1.5) to [out=270,in=270] (0,1.5) to (0,2);
\end{tikzpicture}
\end{gather}
\end{lemma}
\begin{proof}
As in \cite{MR2647055}.
\end{proof}

\begin{lemma}
  \label{lem:regionsimpl}
  For every $n\geq 1$, there is an isomorphism in $\foam$ that simplifies the
coherently oriented (clockwise or anti-clockwise) $2n$-gon web $W_n$, e.g.:
\[
  \begin{tikzpicture}[anchorbase, scale=.5]
    \draw [very thick] (.5,0) -- (.5,.5);
    \draw [double, directed=0.55] (.5,.5) .. controls (.4,.35) and (0,.6) .. (0,1) .. controls (0,1.4) and (.4,1.65) .. (.5,1.5);
    \draw [very thick,rdirected=.55] (.5,.5) .. controls (.6,.35) and (1,.6) .. (1,1) .. controls (1,1.4) and (.6,1.65) .. (.5,1.5);
    \draw [very thick, ->] (.5,1.5) -- (.5,2);
    \end{tikzpicture} 
    \quad\overset{n=1}{\cong} \quad
    \begin{tikzpicture}[anchorbase, scale=.5]
    \draw [very thick,->] (.5,0) -- (.5,2);
    \end{tikzpicture}   
    \quad,\quad
    \begin{tikzpicture}[anchorbase, scale=.5]
    \draw [very thick, ->] (-.5,0) to (0,.5) (0,1.5) to (-.5,2);
    \draw [double] (0,.5) to (0,1.5);
    \draw [very thick,->] (1.5,2) to (1,1.5) (1,0.5) to (1.5,0);
    \draw [double] (1,.5) to (1,1.5);
    \draw [very thick, directed=.65] (0,1.5) to (1,1.5);
    \draw [very thick, directed=.65] (1,0.5) to (0,0.5);
    \end{tikzpicture} 
    \quad\overset{n=2}{\cong} \quad
    \begin{tikzpicture}[anchorbase, scale=.5]
    \draw [very thick,->] (-.5,0) to [out=45,in=135] (1.5,0);
    \draw [very thick,->] (1.5,2) to [out=215,in=315] (-.5,2);
    \end{tikzpicture}     
    \quad,\quad
    \begin{tikzpicture}[anchorbase, scale=.5]
    \draw [very thick, ->] (-.5,0) to (0,.5) (0,1.5) to (-.5,2);
    \draw [double] (0,.5) to (0,1.5);
    \draw [very thick,->] (1,2) to (1,1.5) (1,0.5) to (1,0);
    \draw [double] (1,.5) to (2,.5);
    \draw [double] (1,1.5) to (2,1.5);
    \draw [very thick, ->] (2.5,0) to (2,.5) (2,1.5) to (2.5,2);
    \draw [very thick, directed=.65] (0,1.5) to (1,1.5);
    \draw [very thick, directed=.65] (1,0.5) to (0,0.5);
    \draw [very thick, directed=.65] (2,1.5) to (2,0.5);
    \end{tikzpicture} 
    \quad\overset{n=3}{\cong}\quad
    \begin{tikzpicture}[anchorbase, scale=.5]
    \draw [very thick,->] (-.5,0) to (-.25,0.25) to [out=45,in=90] (1,.15) to (1,0);
    \draw [very thick,->] (1,2) to (1,1.75) to [out=270,in=315] (-.25,1.85) to (-.5,2);
    \draw [very thick, ->] (2.5,0) to [out=135,in=270] (2,1) to [out=90,in=215] (2.5,2);
    \end{tikzpicture} 
\]
The isomorphism can be chosen to be a foam $F_n$, such that $c(F_n)=c(W_n)\times
[0,1]$ and such that $F_n$ contains a single $2$-labeled disk in the region shown.
\end{lemma}
\begin{proof}
  Use the local isomorphisms from \eqref{eq:squares} to undock the 2-labeled edges from the boundary edges, leaving a central 2-labeled circle, which can then be removed
  via \eqref{eq:circles}.
\end{proof}

The isomorphisms in Lemma~\ref{lem:regionsimpl} simplify webs by removing coherently (anti-) clockwise oriented regions by undocking 2-labeled edges on the left (right) side of the remaining 1-labeled edges. The following lemma shows that this is possible for every closed web, if we also allow exceptional cases of the following type:

\[
\begin{tikzpicture}[xscale =-.5, smallnodes,yscale=.5,rotate=270,anchorbase]
  \draw[very thick] (0,2.25) \ur (1.5,2.75) \rd (3,2.25) to (3,-3.25) \dl(1.5,-3.75) \lu (0,-3.25) to [out=90,in=210] (.5,-2.5);
  \draw[very thick] (1,-2.75) to  (.5,-2.5);
  \draw[double] (.5,-1.75) to (.5,-2.5);
  \draw[very thick] (0,-1)  to [out=270,in=150] (.5,-1.75);
  \draw[very thick] (1,-1.5) to  (.5,-1.75);
  \draw[very thick, dotted] (0,-1) to (0,0);
  \draw[very thick] (0,0) to [out=90,in=210] (.5,.75);
  \draw[very thick] (1,.5) to  (.5,.75);
  \draw[double] (.5,.75) to (.5,1.5);
  \draw[very thick,<-] (0,2.25)  to [out=270,in=150] (.5,1.5);
  \draw[very thick] (1,1.75) to  (.5,1.5);
  \draw (-.25,-.5) \ur (.5,0) \ru (1,.5) to  (1,1.75) \ur (1.75,2.25) \rd (2.5,1.75) to (2.5,-2.75) \dl (1.75,-3.25) \lu (1,-2.75) to (1,-1.75) \ul (.5,-1) \lu (-.25,-.5) ;
  \node at (1.75,-.5) {$W$};
\end{tikzpicture}  
\longrightarrow
\begin{tikzpicture}[xscale =-.5, smallnodes,yscale=.5,rotate=270,anchorbase]
  \draw[double,->] (-.5,2.25) \ur (1.5,3.25) \rd (3.5,2.25) to (3.5,-3.25) \dl(1.5,-4.25) \lu (-.5,-3.25) to (-.5,2.25);
  \draw[very thick,<-] (0,2.25) \ur (1.5,2.75) \rd (3,2.25) to (3,-3.25) \dl(1.5,-3.75) \lu (0,-3.25);
  \draw[very thick] (1,-2.75) to [out=150,in=90]  (0,-3.25);
  \draw[very thick] (0,-1)  to [out=270,in=210] (1,-1.5);
  \draw[very thick, dotted] (0,-1) to (0,0);
  \draw[very thick] (0,0) to [out=90,in=150] (1,.5);
  \draw[very thick] (0,2.25)  to [out=270,in=210] (1,1.75);
  \draw (-.25,-.5) \ur (.5,0) \ru (1,.5) to  (1,1.75) \ur (1.75,2.25) \rd (2.5,1.75) to (2.5,-2.75) \dl (1.75,-3.25) \lu (1,-2.75) to (1,-1.75) \ul (.5,-1) \lu (-.25,-.5) ;
  \node at (1.75,-.5) {$W$};
\end{tikzpicture}  
\]

Here we consider the outside region as a coherently clockwise oriented and perform undocking towards the left. The resulting 2-labeled outside circle can be removed by a cap foam after simplifying the remaining web nested inside it, as described in the following lemma.

\begin{lemma}
  \label{lem:existregion} Every closed $\glnn{2}$ web $W$ is isomorphic to a direct sum of grading shifts of the empty web via an isomorphism which is composed of exclusively clockwise local isomorphisms from Lemma~\ref{lem:regionsimpl} (including `outside' versions) and the circle removal isomorphisms \eqref{eq:circles}. As such, this isomorphism is unique.
  The statement with exclusively anti-clockwise isomorphisms holds analogously.
\end{lemma}
\begin{proof} The statement is trivial for webs $W$ without trivalent vertices, so let us assume that $W$ has trivalent vertices. Moreover, we will assume that $W$ is connected, for otherwise we can treat the connected components independently, starting with an innermost one with respect to the nesting condition. Differences in ordering these simplifications give the same result due to far-commutativity.

Fix a base point $p\in \R^2\setminus W$ in the outside region  . We now start by labeling each region in $\R^2\setminus W$ with its `winding number' around $p$. For this, pick a point $q$ in the region and a path $\gamma\colon [0,1]\longrightarrow \R^2$ with $\gamma(0)=q$ and $\gamma(1)=p$. Then we label the region containing $q$ by the algebraic intersection number of $\gamma$ with $W$. (Here each signed intersection with the $2$-labeled edges counts as $\pm 2$.) E.g. the region enclosed by a clockwise circle will have winding number $1$. The outside region has winding number $0$, by definition. 

Now we have two cases to consider. If the maximal winding number of regions is positive, then each maximal region has a coherently clockwise oriented boundary and can be removed via Lemma~\ref{lem:regionsimpl}. 

Otherwise, the maximal winding number is zero, and the outside region attains this maximum. This implies that the outermost cycle in the web (the boundary of the region containing $p$) is coherently counter-clockwise oriented. In this case, we use the `outside version' of the undocking move to simplify the web. 

This algorithm terminates since it strictly reduces the number of trivalent vertices in every step. Since regions of maximal winding number cannot be adjacent, the order of simplifying is again irrelevant due to far-commutativity.
\end{proof}

Lemma~\ref{lem:existregion} allows us to build a basis for the vector space $\HOM_{\foam}(\emptyset, W)$ for any web $W$. 
One of the possible variations still allowed are rescalings of the circle removal isomorphisms from \eqref{eq:circles} by signs. To simplify local computations we choose a rescaling depending of the auxiliary data of a flow, that we now introduce. 

\begin{definition} A \emph{flow} $f$ with values $\{+1,-1\}$ \emph{on a web} $W$ in $\foam$ assigns to every 1-labeled
edge of $W$ an element of $\{+1,-1\}$ and to every 2-labeled edge of $W$ the entire set
$\{+1,-1\}$, such that every trivalent vertex in $W$ has adjacent labels as in the following:
\[
  \begin{tikzpicture}[scale =.5, smallnodes,yscale=1]
    \draw[very thick] (0,0) node[below,xshift=-2pt]{$+1$} to [out=90,in=210] (.5,.75);
    \draw[very thick] (1,0) node[below,xshift=2pt]{$-1$} to [out=90,in=330] (.5,.75);
    \draw[double,->] (.5,.75) to (.5,1.5) node[above]{$\{+1,-1\}$};
  \end{tikzpicture}  
  \quad,\quad
  \begin{tikzpicture}[scale =.5, smallnodes,yscale=1]
    \draw[very thick] (0,0) node[below,xshift=-2pt]{$-1$} to [out=90,in=210] (.5,.75);
    \draw[very thick] (1,0) node[below,xshift=2pt]{$+1$} to [out=90,in=330] (.5,.75);
    \draw[double,->] (.5,.75) to (.5,1.5) node[above]{$\{+1,-1\}$};
  \end{tikzpicture}  
  \quad,\quad
  \begin{tikzpicture}[scale =.5, smallnodes,yscale=-1]
    \draw[very thick,<-] (0,0) node[above,xshift=-2pt]{$+1$} to [out=90,in=210] (.5,.75);
    \draw[very thick,<-] (1,0) node[above,xshift=2pt]{$-1$} to [out=90,in=330] (.5,.75);
    \draw[double] (.5,.75) to (.5,1.5) node[below]{$\{+1,-1\}$};
  \end{tikzpicture}  
  \quad,\quad
  \begin{tikzpicture}[scale =.5, smallnodes,yscale=-1]
    \draw[very thick,<-] (0,0) node[above,xshift=-2pt]{$-1$} to [out=90,in=210] (.5,.75);
    \draw[very thick,<-] (1,0) node[above,xshift=2pt]{$+1$} to [out=90,in=330] (.5,.75);
    \draw[double] (.5,.75) to (.5,1.5) node[below]{$\{+1,-1\}$};
  \end{tikzpicture} 
\]
I.e. the \emph{flow condition} is satisfied at every vertex. Similarly, a \emph{flow
on a foam} associates to 1-labeled facets elements of $\{+1,-1\}$ and to
2-labeled facets the entire set $\{+1,-1\}$, such that the flow condition is
satisfied at every seam. A flow on a foam induces a flow on its boundary webs.
When considering a foam $F\colon W\longrightarrow V$ between webs with flows $f_W$ and
$f_V$, then we say these flows are \emph{compatible} with $F$ if there exists a flow $f$
on $F$ that restricts to $f_W$ and $f_V$ on the corresponding boundary
components.

We say a flow on a web $W$ is \emph{admissible} if every non-closed 2-labeled edge has
a neighborhood of the following two types:
\begin{equation}
\label{eqn:adm-res}
  \begin{tikzpicture}[scale =.5, smallnodes,yscale=1]
    \draw[very thick] (0,0) node[below,xshift=-2pt]{$+1$} to [out=90,in=210] (.5,.75);
    \draw[very thick] (1,0) node[below,xshift=2pt]{$-1$} to [out=90,in=330] (.5,.75);
    \draw[double] (.5,.75) to (.5,1.5);
    \draw[very thick,<-] (0,2.25) node[above,xshift=-2pt]{$+1$} to [out=270,in=150] (.5,1.5);
    \draw[very thick,<-] (1,2.25) node[above,xshift=2pt]{$-1$} to [out=270,in=30] (.5,1.5);
  \end{tikzpicture}  
  \quad,\quad
  \begin{tikzpicture}[scale =.5, smallnodes,yscale=1]
    \draw[very thick] (0,0) node[below,xshift=-2pt]{$-1$} to [out=90,in=210] (.5,.75);
    \draw[very thick] (1,0) node[below,xshift=2pt]{$+1$} to [out=90,in=330] (.5,.75);
    \draw[double] (.5,.75) to (.5,1.5);
    \draw[very thick,<-] (0,2.25) node[above,xshift=-2pt]{$-1$} to [out=270,in=150] (.5,1.5);
    \draw[very thick,<-] (1,2.25) node[above,xshift=2pt]{$+1$} to [out=270,in=30] (.5,1.5);
  \end{tikzpicture}  
\end{equation}
In words, the flow has to stay parallel, and not cross, at 2-labeled edges.
\end{definition}

\begin{lemma}
\label{lem:link2flow}Let $L$ be an oriented link diagram in $\R^2$, which we consider as a union of arcs that connect crossings and disjoint closed components. Next we color the connected components of $\R^2\setminus L$ with the checkerboard black and white coloring, starting with white on the unique unbounded component. On every oriented arc and every closed component of $L$, we now place the label $+1$ if there is a white component on the left, and the label $-1$ otherwise. This labelling descends to an admissible flow on every web $W$ in the cube of resolutions of $L$, compatible with the foams representing the components of the differential. 

Moreover, this assignment of a canonical flow to each link diagram in $\R^2$ is compatible with split disjoint union of link diagrams.
\end{lemma}
\begin{proof}
At the site of every crossing, the webs $W$ contain either a parallel resolution or a 2-labelled edge. In the latter case, the labels are in the configurations of \eqref{eqn:adm-res}, depending on whether there is a white or black region to the left of this crossing in $L$. In the case of a parallel resolution, the situation is analogous. Thus, the labels assemble to flows which are admissible at 2-labelled edges and compatible with the zip and unzip foams realising the components of the differential in the cube of resolutions. 
  
  The statement about compatibility of the flow construction with split disjoint union follows from the fact that this operation embeds two white outside regions into a white outside region.
\end{proof}

\begin{remark}
More generally, using the same arguments as in Lemma~\ref{lem:link2flow} one can show that any foam between closed $\glnn{2}$ webs in $\R^2\times [0,1]$ admits a canonical flow.
\end{remark}

\begin{definition} \label{basis def} Let $W$ be a $\glnn{2}$ web with an admissible flow $f$. We define a basis $B(W,f)$ of
$\HOM_{\foam}(\emptyset, W)$ containing (signed) foams appearing as entries of
the inverse of the isomorphism obtained in Lemma~\ref{lem:existregion}. These foams are determined by the local
simplifications in Lemma~\ref{lem:regionsimpl}, where we only use clockwise
regions, and the isomorphisms realising \eqref{eq:circles}. To determine the
signs, note that the flow $f$ on $W$ extends uniquely to a flow on every basis foam,
which we again denote by $f$. We now equip the basis foams
with a minus sign if they contain an odd number of dots on facets where $f$ has value
$-1$. 
\end{definition}

\begin{example}\label{exa:circle} The bases for 1-labeled circles with flow $+1$ and $-1$ are:
  \[
    B(\bigcirc,+1)=\Big\{\,
    \begin{tikzpicture} [fill opacity=0.2, tinynodes,  decoration={markings, mark=at position 0.6 with {\arrow{>}};  }, scale=.4,anchorbase]
      \draw [fill=blue] (0,4) ellipse (1 and 0.5);
      \draw (-1,4) .. controls (-1,2) and (1,2) .. (1,4);
      \path [fill=blue, opacity=.3] (1,4) .. controls (1,3.34) and (-1,3.34) .. (-1,4) --
        (-1,4) .. controls (-1,2) and (1,2) .. (1,4);
         \node[opacity=1] at (0.2,3) {\tiny $+1$};
    \end{tikzpicture}  
    \,
    ,
    \,
    \begin{tikzpicture} [fill opacity=0.2,  decoration={markings, mark=at position 0.6 with {\arrow{>}};  }, scale=.4,anchorbase]
      \draw [fill=blue] (0,4) ellipse (1 and 0.5);
      \draw (-1,4) .. controls (-1,2) and (1,2) .. (1,4);
      \path [fill=blue, opacity=.3] (1,4) .. controls (1,3.34) and (-1,3.34) .. (-1,4) --
        (-1,4) .. controls (-1,2) and (1,2) .. (1,4);
       \node[opacity=1] at (0,4) {$\bullet$};
       \node[opacity=1] at (0.2,3) {\tiny $+1$};
    \end{tikzpicture}  
    \,
    \Big\}
    \quad,\quad
    B(\bigcirc,-1)=\Big\{\,
    \begin{tikzpicture} [fill opacity=0.2,  decoration={markings, mark=at position 0.6 with {\arrow{>}};  }, scale=.4,anchorbase]
      \draw [fill=blue] (0,4) ellipse (1 and 0.5);
      \draw (-1,4) .. controls (-1,2) and (1,2) .. (1,4);
      \path [fill=blue, opacity=.3] (1,4) .. controls (1,3.34) and (-1,3.34) .. (-1,4) --
        (-1,4) .. controls (-1,2) and (1,2) .. (1,4);
      \node[opacity=1] at (0.2,3) {\tiny $-1$};
    \end{tikzpicture}  
    \,
    ,
    -
    \,
        \begin{tikzpicture} [fill opacity=0.2,  decoration={markings, mark=at position 0.6 with {\arrow{>}};  }, scale=.4,anchorbase]
      \draw [fill=blue] (0,4) ellipse (1 and 0.5);
      \draw (-1,4) .. controls (-1,2) and (1,2) .. (1,4);
      \path [fill=blue, opacity=.3] (1,4) .. controls (1,3.34) and (-1,3.34) .. (-1,4) --
        (-1,4) .. controls (-1,2) and (1,2) .. (1,4);
       \node[opacity=1] at (0,4) {$\bullet$};
       \node[opacity=1] at (0.2,3) {\tiny $-1$};
    \end{tikzpicture}  
    \,
    \Big\}
  \]
We will write $\oone$ and $\ox$ for the undotted and dotted cup foam, respectively.
Both bases above can then be written as $\{\oone, \epsilon \ox\}$.
\end{example}

\begin{example} For the theta web, there again exist two flows, whose associated bases are:
  \[
    B\Big(\!
      \begin{tikzpicture} [fill opacity=0.2, scale=.4,,tinynodes,anchorbase]
       \draw[very thick] (0,0) circle (1);
       \draw[double,directed=.65] (0,-1) to (0,1);
       \node[opacity=1] at (-1.55,0) {$+1$};
       \node[opacity=1] at (1.55,0) {$-1$};
      \end{tikzpicture}  \!
    \Big)
    =\Big\{\,
    \begin{tikzpicture} [fill opacity=0.2,  decoration={markings, mark=at position 0.6 with {\arrow{>}};  }, scale=.4,anchorbase]
      \draw [very thick, blue] (.8,3.7) .. controls (.75,2.2) and (-.75,2.2) .. (-.8,4.3);
      \fill [yellow,opacity=0.5] (.8,3.7) .. controls (.75,2.2) and (-.75,2.2) .. (-.8,4.3);
      \draw [fill=blue] (0,4) ellipse (1 and 0.5);
      \draw (-1,4) .. controls (-1,2) and (1,2) .. (1,4);
      \draw [double,<-] (-.8,4.3) to (.8,3.7);
      \path [fill=blue, opacity=.3] (1,4) .. controls (1,3.34) and (-1,3.34) .. (-1,4) --
        (-1,4) .. controls (-1,2) and (1,2) .. (1,4);
        \node[opacity=1] at (0.2,3) {\tiny $+1$};
    \end{tikzpicture}  
    \,
    ,
    \,
    \begin{tikzpicture} [fill opacity=0.2,  decoration={markings, mark=at position 0.6 with {\arrow{>}};  }, scale=.4,anchorbase]
      \draw [very thick, blue] (.8,3.7) .. controls (.75,2.2) and (-.75,2.2) .. (-.8,4.3);
      \fill [yellow,opacity=0.5] (.8,3.7) .. controls (.75,2.2) and (-.75,2.2) .. (-.8,4.3);
      \draw [fill=blue] (0,4) ellipse (1 and 0.5);
      \draw (-1,4) .. controls (-1,2) and (1,2) .. (1,4);
      \draw [double,<-] (-.8,4.3) to (.8,3.7);
      \path [fill=blue, opacity=.3] (1,4) .. controls (1,3.34) and (-1,3.34) .. (-1,4) --
        (-1,4) .. controls (-1,2) and (1,2) .. (1,4);
        \node[opacity=1] at (-.6,3.2) {$\bullet$};
        \node[opacity=1] at (0.2,3) {\tiny $+1$};
    \end{tikzpicture}  
    \,
    \Big\}
    \quad,\quad
    B\Big(\!
      \begin{tikzpicture} [fill opacity=0.2, scale=.4,,tinynodes,anchorbase]
       \draw[very thick] (0,0) circle (1);
       \draw[double,directed=.65] (0,-1) to (0,1);
       \node[opacity=1] at (-1.55,0) {$-1$};
       \node[opacity=1] at (1.55,0) {$+1$};
      \end{tikzpicture}  \!
    \Big)
    =\Big\{\,
    \begin{tikzpicture} [fill opacity=0.2,  decoration={markings, mark=at position 0.6 with {\arrow{>}};  }, scale=.4,anchorbase]
      \draw [very thick, blue] (.8,3.7) .. controls (.75,2.2) and (-.75,2.2) .. (-.8,4.3);
      \fill [yellow,opacity=0.5] (.8,3.7) .. controls (.75,2.2) and (-.75,2.2) .. (-.8,4.3);
      \draw [fill=blue] (0,4) ellipse (1 and 0.5);
      \draw (-1,4) .. controls (-1,2) and (1,2) .. (1,4);
      \draw [double,<-] (-.8,4.3) to (.8,3.7);
      \path [fill=blue, opacity=.3] (1,4) .. controls (1,3.34) and (-1,3.34) .. (-1,4) --
        (-1,4) .. controls (-1,2) and (1,2) .. (1,4);
        \node[opacity=1] at (0.2,3) {\tiny $-1$};
    \end{tikzpicture}  
    \,
    ,
    -
    \,
    \begin{tikzpicture} [fill opacity=0.2,  decoration={markings, mark=at position 0.6 with {\arrow{>}};  }, scale=.4,anchorbase]
      \draw [very thick, blue] (.8,3.7) .. controls (.75,2.2) and (-.75,2.2) .. (-.8,4.3);
      \fill [yellow,opacity=0.5] (.8,3.7) .. controls (.75,2.2) and (-.75,2.2) .. (-.8,4.3);
      \draw [fill=blue] (0,4) ellipse (1 and 0.5);
      \draw (-1,4) .. controls (-1,2) and (1,2) .. (1,4);
      \draw [double,<-] (-.8,4.3) to (.8,3.7);
      \path [fill=blue, opacity=.3] (1,4) .. controls (1,3.34) and (-1,3.34) .. (-1,4) --
        (-1,4) .. controls (-1,2) and (1,2) .. (1,4);
        \node[opacity=1] at (-.6,3.2) {$\bullet$};
        \node[opacity=1] at (0.2,3) {\tiny $-1$};
    \end{tikzpicture}  
    \,
    \Big\}
  \]
  Let us call the theta web with the first flow the \emph{left theta web} and
  the other one the \emph{right theta web}.
\end{example}

\begin{example} The unzip foam from the left theta web to two circles has all its
coefficients in the set $\{0,1\}$ in the above bases. This is seen using the neck-cutting relation (\ref{sl2closedfoam}) and the last relation in equation (\ref{sl2neckcutting_enh_2lab}). For the right theta web, it has
coefficients $\{0,-1\}$. 

The zip foam with target the theta web (with any flow) and
  source given by two circles with compatible flows has coefficients $\{0,1\}$
  in the above bases. 
\end{example}

\subsection{Action of foams on bases}
Here we study the action of elementary foams on the web bases $B(W,f)$ constructed in the previous section. To this end, we will need two auxiliary results.

\begin{proposition}(Preparing generalized neck-cutting)
\label{prop:preneckcut}
Let $F$ be a foam in $\R^2\times [0,1]$ whose underlying 1-labeled surface $c(F)$ has a compression disk $D$. Then there exist tubular neighbourhoods $D \subset U \subset V \subset \R^2 \times (0,1)$ and a foam $G$ in $\R^2\times [0,1]$ such that 
\begin{itemize}
    \item $G\cap (\R^2\setminus V)=F\cap (\R^2\setminus V)$
    \item $c(G)=c(F)$
    \item $c(F\cap V)\cong \partial D \times [0,1]$
    \item $G\cap U \cong \partial D \times [0,1]$
    \item $F=\pm G $ as morphisms in $\foam$
\end{itemize}
\end{proposition}
To paraphrase this: any foam, whose underlying 1-labeled surface has a neck, can be modified locally and up to sign into a new foam, which  itself has a 1-labeled neck.  
\begin{proof}
This was shown in the proof of \cite[Lemma 3.6]{1806.03416}, so we only give a sketch here. Assuming that $D$ intersects $F$ generically, we can find a slightly larger open disk $D'$ and a tubular neighbourhood $V\cong D' \times (0,1)$, such that $F\cap V\cong W\times (0,1)$ for some $\glnn{2}$ web $W$ with $c(W)=\partial D $. Using the 2-labeled circle, saddle, and undocking relations from \ref{lem:webisos}, one can see that the foam $W \times (0,1)$ equals up to a sign a foam $G'$ that factors through a web $W'$ with $\partial W'=\partial W$, $c(W')=c(W)$ and $D\cap W'=\partial D$. Detailed descriptions of this process appear in \cite[Lemmas 2.1 and 3.6]{1806.03416} and \cite[Lemma 64]{2018arXiv180309883Q}. The foam $G$ is now obtained by gluing $G'$ and $F\cap(\R^2\setminus V)$.
\end{proof}

\begin{proposition}
\label{prop:agree}
If $F$ and $G$ are (possibly dotted) Blanchet foams with identical underlying (dotted) surfaces $c(F)=c(G)$ and $\partial F = \partial G$, then $F=\pm G$.
\end{proposition}
\begin{proof}
This was proven in \cite[Proposition 2.9]{2019arXiv190312194B}. Here we sketch an alternative proof based on Proposition~\ref{prop:preneckcut}. First we may assume that $W=\partial F= \partial G \subset \R^2 \times \{1\}$, for otherwise we would bend the bottom boundary of $F$ and $G$ to the top. 

Fix an admissible flow $f$ on $W$ and consider the basis $B(W,f)\subset \HOM_{\foam}(\emptyset,W)$. By construction, all elements of $B(W,f)$ are given by decorating the 1-labeled components of a distinguished foam $F^W$ by certain signed dots. We will also consider a basis $B(W,f)^*$ of $\HOM_{\foam}(W,\emptyset)$ which is dual to $B(W,f)$ under the composition pairing valued in $\END_{\foam}(\emptyset) = \Z\langle \emptyset \rangle \cong \Z$. More specifically, we take $B(W,f)^*$ to consist of foams obtained by appropriately decorating a distinguished foam $F_W$ with signed dots ($F_W$ can be obtained by reflecting $F^W$ in $\R^2\times \{1\}$. 

Now we claim, there exists $\epsilon\in \{\pm 1\}$ such that $H \circ F = \epsilon (H \circ G)$ for any $H \in B(W,f)^*$. This would imply $F=\epsilon G$.

To prove the claim, we choose base points on 1-labeled facets of $F_W$, such that every non-closed
connected component of $c(F)=c(G)$
contains exactly one such base point up inclusion into $c(F_W\circ F)=c(F_W\circ G)$. The foams $H\circ F$ and $H \circ G$ are signed, dotted versions of $F_W\circ F$ and $F_W \circ G$ respectively, and to verify their equality modulo foams relations up to a global sign, we may assume that all dots coming from $H$ and from non-closed components of $c(F)=c(G)$ have been transported to the base points. This is possible, since the dots originating in $F$ and $G$ are moved to the base points at the expense of a global sign $\epsilon'$, irrespective of $H$.

We have now reduced the problem to a comparison of foams
\[ C\circ d_H \circ F' \quad \text{ and } \quad \epsilon' C\circ d_H \circ G'\]
where $C$ consists entirely of 1-labeled caps, $d_H$ is a signed, dotted identity foam on $\partial C$, with signs and dots depending on $H \in B(W,f)^*$, $F'$ and $G'$ are foams satisfying $c(F')=c(G')$ and containing dots only on closed components of $c(F')=c(G')$. In particular, we have eliminated the sign dependence on $H$ and are left with proving $F'=\pm G'$. 

Since $c(F')=c(G')$ we have the same necks in $F'$ and $G'$, which can be cut using Proposition~\ref{prop:preneckcut} and the neck-cutting relation from \eqref{sl2closedfoam} at the expense of a global sign. Similarly, every $S^2$ component in the underlying surfaces of the result can be removed by another application of \ref{prop:preneckcut} and the sphere relations from \eqref{sl2closedfoam}. See \cite[Lemma 3.9]{1806.03416} for a detailed description of this step. After these simplifications, we may assume that $F'=\epsilon_F \sum_i  d_i F''$ and $G'= \epsilon_G \sum_i d_i G''$ where $F''$ and $G''$ are undotted foams such that $c(F'')=c(G'')$ consists entirely of cups (disks), $d_i$ is a dotted identity foam on $\partial C$ and $\epsilon_F, \epsilon_G\in \{\pm 1\}$. Here we have used in a crucial way that the neck-cutting relation from \eqref{sl2closedfoam} is sign-coherent and always involves foam facets with the same flow label. Finally, we observe that $F''= \pm G''$ by undocking and eliminating all 2-labeled facets in both foams via the first two relations in \eqref{sl2neckcutting_enh_2lab} and the first relation in \eqref{sl2Fig5Blanchet_1}.
\end{proof}

Next, we compute the action of zips, unzips, and saddle maps between standard bases. Note that a saddle cobordism between link diagrams $L$ and $L'$ induces saddle foams between webs $W$ and $W'$ appearing in the respective cubes of resolutions, which are compatible with the canonical flows defined in Lemma~\ref{lem:link2flow}.

\begin{proposition}\label{prop:zipunzip} Let $W$ be a web with an admissible
flow $f$, and let $F\colon W\longrightarrow V$ denote a zip or an unzip foam  or a saddle to a web $V$,
which we equip with the induced flow, denoted $g$. Then the matrix of the linear
map 
\[F_*\colon \HOM_{\foam}(\emptyset, W) \longrightarrow \HOM_{\foam}(\emptyset, V),\qquad F_*(G) = F\circ G \]
with respect to the bases $B(W,f)$ and $B(V,g)$ has entries in either $\{0,1\}$ or $\{0,-1\}$. 
\end{proposition}
This says that the action of $F_*$ is \emph{sign-coherent} with respect to the
standard bases.

\begin{proof} Recall that every basis element $H\in B(W,f)$ is a signed, dotted version of a distinguished foam $F^W\in\HOM_{\foam}(\emptyset,W)$ such that $c(F^W)$ is a collection of disks. Let us write $H=d_H\circ F^W$, where $d_H$ is a signed, dotted identity foam on $W$, with signs determined by the flow $f$. Then we have:
\[
F\circ H = F \circ d_H\circ F^W = d'_H\circ F\circ F^W
\]
for a signed, dotted identity foam on $V$ with signs determined by the flow $g$.

Suppose that $c(F)$ is a saddle that merges two circles in $c(W)$ (recall that $F$ itself may be a zip, an unzip, or a saddle). Then we have $c(F\circ F^W)=c(F^V)$ and $\partial(F\circ F^W)=\partial F^V$, and thus by Proposition~\ref{prop:agree}, $F\circ F^W=\epsilon\cdot F^V$ for some $\epsilon\in\{\pm 1\}$. Furthermore, if $d'_H\circ F\circ F^W= \epsilon\cdot d'_H\circ F^V$ is non-zero, then it agrees with an element of $B(V,g)$ up to the global sign $\epsilon$.

Now suppose that $c(F)$ is a saddle that splits two circles in $c(W)$. In this case, we can find a compression disk $D$ for $c(F\circ F^W)$ to which we apply Proposition~\ref{prop:preneckcut} and the neck-cutting relation from \eqref{sl2closedfoam}. The result will be a linear combination $\epsilon'( d_1\circ G^V+d_2\circ G^V)$ where $\epsilon'\in \{\pm 1\}$, $G^V$ is an undotted foam with $c(G^V)=c(F^V)$, and $d_1$ and $d_2$ are both identity foams on $V$ with a single dot placed on certain facets with the same flow label $\epsilon''\in \{\pm 1\}$. By Proposition~\ref{prop:agree} we conclude that $G^V=\epsilon''' F^V$ for some $\epsilon'''\in \{\pm 1\}$, and thus $F\circ F_W = \epsilon'\epsilon''' (d_1\circ F^V+d_2\circ F^V)$. Finally, for other basis elements we compute $d'_H\circ F\circ F^W = \epsilon'\epsilon''' (d'_H \circ d_1\circ F^V+d'_H \circ d_2\circ F^V)$. Whenever these summands are nonzero, they agree with a basis element from $B(V,g)$ up to the global sign $\epsilon'\epsilon''\epsilon'''$.
\end{proof}

Finally, we consider the effect of the cup and cap cobordisms.

\begin{lemma} \label{cup cap}
Let $W$ be a web with an admissible
flow $f$, containing an innermost 1-labelled circle C with flow label $\epsilon$. Let $F$ denote the cap foam that removes this component, resulting in a web $V:=W\setminus C$ that we equip with the induced flow $g$. Further, let $F^{!} \colon V \longrightarrow W$ denote the corresponding cup foam in the opposite direction.    
Then the matrix of the linear
map 
\[F_*\colon \HOM_{\foam}(\emptyset, W) \longrightarrow \HOM_{\foam}(\emptyset, V),\qquad F_*(G) = F\circ G \]
with respect to the bases $B(W,f)$ and $B(V,g)$ has entries in $\{0,\epsilon\}$ and the matrix of the linear
map 
\[F^{!}_*\colon \HOM_{\foam}(\emptyset, V) \longrightarrow \HOM_{\foam}(\emptyset, W),\qquad F^!_*(G) = F^!\circ G \]
with respect to the bases $B(V,g)$ and $B(W,f)$ has entries in $\{0,1\}$.
\end{lemma}
\begin{proof}
Immediate from Example~\ref{exa:circle} and the relations \eqref{sl2closedfoam}.
\end{proof}

\begin{remark} In this section we chose the additional local data of flows on webs and foams to record and track the sign discrepancies between Bar-Natan cobordisms and Blanchet foams, which are essential for the functoriality of Khovanov homology and link cobordisms. The use of flows is motivated by the Murakami--Ohtsuki--Yamada state-sum for evaluation of closed webs \cite{MR1659228} and its analog for foams that is due to Robert--Wagner \cite{1702.04140}. Other, essentially equivalent ways of encoding these sign discrepancies include skein-algebra like superposition operations with 2-labelled curves \cite{1806.03416} and shadings \cite{2019arXiv190312194B}.
\end{remark}

\begin{remark} \label{deformations} The results of this section also hold for the deformed foam categories mentioned in Remark~\ref{rem:def}. When deforming the relation $\bullet^2=0$ in \eqref{sl2closedfoam} to $\bullet^2 -h\bullet +t=0$, the signed dots $-\bullet$ in basis elements should be replaced by the linear combination $\circ:=h-\bullet$, see \cite{2019arXiv190312194B}.  
\end{remark}

\section{The stable homotopy type} \label{stable homotopy sec}

\subsection{The signed Burnside category and the construction of ${\mathcal X}_{\rm or}(L)$} \label{signed Burnside sec}

A stable homotopy refinement of Khovanov homology  was originally introduced by Lipshitz and Sarkar in \cite{MR3230817}. Their construction was reformulated by Lawson, Lipshitz and Sarkar \cite{MR4153651} using the Burnside category\footnote{See also \cite{MR3465966}}.
In this section we formulate the stable homotopy type ${\mathcal X}_{\rm or}(L)$ in Theorem \ref{main theorem:stable homotopy type}. Our construction relies on the {\em signed Burnside category}, introduced by Sarkar, Scaduto and Stoffregen in \cite{MR4078823} in their work on the odd Khovanov homotopy type.

Recall from Section \ref{Blanchet review sec} the setup of the Blanchet theory: given an oriented link diagram $L$ with $n$ crossings, there is a cube of resolutions $\bracor{L}$ and the chain complex of graded abelian groups $\CKhor(L)$.
More precisely, to each vertex $v$ of the cube $\{0,1\}^n$ there is an associated planar web $W(v)$, and each edge corresponds to a zip/unzip foam. Lemma \ref{lem:link2flow} gives a canonical flow $f_{\mathrm{can}}$ on every web $W(v)$ in the cube of resolutions of $L$, compatible with the foams representing the components of the differential. A basis
 $B(v)$ for the web space at each vertex $v$ is constructed in Definition \ref{basis def}. Finally, Proposition \ref{prop:zipunzip} determines the coefficients of the edge differential with respect to the chosen bases.

Several variations of the stable homotopy type construction using different versions of the Burnside category have appeared in the literature; therefore we will outline steps of the construction rather than giving a detailed exposition. We will refer to statements in \cite{MR4153651, MR4078823} and emphasize details of our setting that are different from these references.

Consider the cube category ${\underline{2}}^n$ whose objects are elements of $\{0, 1\}^n$ and with a unique morphism ${\phi}_{a,b}$
iff $a \geq  b$ (here the partial ordering on the objects is induced by the ordering of the coordinates). The construction is based on the lift $F_{\rm or}$

\begin{equation} \label{eq:lift}
\begin{tikzcd}
& \mathscr{B}_{\sigma} \arrow[d] \\
{\underline{2}}^n \arrow[r, "\scriptsize\mathfrak{F}^{\rm op}_{\rm or}" '] \arrow[ur, dashrightarrow, "\scriptsize F_{\rm or}"]
& {\mathbb Z}{\rm -}{\rm Mod} 
\end{tikzcd}
\end{equation}
to the signed Burnside category $\mathscr{B}_{\sigma}$ (the definition is recalled below) of the cube of resolutions, viewed as a diagram of abelian groups, $\mathfrak{F}^{\rm op}_{\rm or}\! : {\underline{2}}^n \longrightarrow {\mathbb Z}{\rm -}{\rm Mod}$. 
Here the subscript stands for ``oriented'', and the superscipt reflects the fact that the arrows in ${\underline{2}}^n$ are dual to those in the usual cube of resolutions. 
The value of $\mathfrak{F}^{\rm op}_{\rm or}$ on vertices and edges of the cube is defined by the representable functor $\HOM_{\foam}(\emptyset,-)$ as in equation (\ref{Hom eq}).
The construction of the Khovanov stable homotopy type in \cite{MR4153651} relied on the Burnside category $\mathscr{B}$, a weak 2-category whose objects are finite sets, $1$-morphisms are finite correspondences, and $2$-morphisms are maps of correspondences. We use the signed Burnside category $\mathscr{B}_{\sigma}$, introduced in \cite{MR4078823}, where $1$-morphisms are {\em signed} correspondences discussed in more detail below. This extension is needed to accommodate the signs appearing in the differential in the oriented model for Khovanov homology (see Proposition \ref{prop:zipunzip}).
The vertical map in the diagram (\ref{eq:lift}) is the forgetful functor sending a finite set to the abelian group generated by it.

There is a forgetful functor $\mathscr{B}_{\sigma}\longrightarrow \mathscr{B}$, and the functor $F_{\rm or}$ constructed in this paper is a lift of $F_{Kh}\! : {\underline 2}^n \longrightarrow {\mathscr B}$ of \cite{MR4153651}. (A different lift, $F_o$ in \cite{MR4078823}, corresponds to the odd Khovanov homology.) As in prior constructions, the functor $F_{\rm or}$ decomposes as $\coprod_j F^j_{\rm or}$ along quantum gradings.

Given finite sets $X, Y$, a {\em signed correspondence} \cite[Section 3.2]{MR4078823} is a tuple $(A,s^{}_{A}, t^{}_{A}, {\sigma}^{}_{A} )$ where $A$ is a finite set, the source and target maps $s^{}_{A}\! : A\longrightarrow X, t^{}_{A}\! : A\longrightarrow Y$ are maps of sets, and ${\sigma}^{}_{A}\! : A\longrightarrow \{\pm 1\} $ is a ``sign map''. Signed correspondences are conveniently encoded as diagrams 
\begin{equation} \label{eq:signed correspondence}
\begin{tikzcd}[row sep = scriptsize, column sep = scriptsize]
& \{\pm 1\} &  \\
&  A  \arrow[dl, "s^{}_{A}"'] \arrow[u, "{\sigma}^{}_{A}"] \arrow[dr, "t^{}_{A}"] & \\
X & & Y
\end{tikzcd}
\end{equation}

The following lemma is a useful tool for constructing functors to the signed Burnside category.

\begin{lemma}\emph{(\cite[Lemma 4.4]{MR4153651}, \cite[Lemma 3.2]{MR4078823})}\label{lem:functor}
The following data satisfying conditions (\ref{item 1}), (\ref{item 2}) can be extended to a strictly unitary lax $2$-functor $F:{\underline 2}^n \longrightarrow {\mathscr{B}}_{\sigma}$, which is unique up to natural isomorphism:

A finite set $F(u)$ for each vertex $u\in {\underline 2}^n$, a signed correspondence $F(\varphi_{u,v})$ from $F(u)$ to $F(v)$ for each pair of vertices $u,v\in {\underline 2}^n$ with $u\geq_1 v$, and a $2$-morphism
$F_{u,v,v',w} : F(\varphi_{v,w}) \circ F(\varphi_{u,v}) \longrightarrow F(\varphi_{v',w}) \circ F(\varphi_{u,v'})$
for each $2$-dimensional face of ${\underline 2}^n$ with vertices $u,v,v',w$ satisfying $u\geq_1 v, v' \geq_1 w$.

\begin{enumerate}
\item \label{item 1} $F_{u,v,v',w}^{-1} = F_{u,v',v,w}$
\item \label{item 2} For every $3$-dimensional sub-cube of ${\underline 2}^n$, the hexagon of Figure \ref{hex figure} commutes.
\begin{figure}[ht]\centering
\begin{subfigure}[b]{0.3\textwidth}\[
\begin{tikzcd}[row sep = small, column sep = small]
& v'  \ar{rr} \ar{dd}  & &  w  \ar{dd}  \\
u \ar[crossing over]{rr} \ar{dd} \ar{ur} & & v \ar{ur} \\
& w'  \ar{rr}  & &  z \\
v'' \ar{rr} \ar{ur} && w'' \ar{ur} \ar[from=uu,crossing over]
\end{tikzcd} \]
\end{subfigure}\hskip5em
\begin{subfigure}[b]{0.5\textwidth}
    \[ \begin{tikzcd}[row sep = small, column sep = small, outer sep = 2pt]
     & \circ\ar[rr, "F_{v',w,w',z} \times \rm{id}"] &  & \circ \ar[ddr, "\rm{id} \times F_{u,v',v'',w'}"] &  \\
     &  &  &  &  \\
    \circ \ar[uur, "\emph{\text{id}} \times F_{u,v,v',w}"] \ar[ddr, "F_{v,w,w'',z} \times \rm{id}"'] &  &  &  & \circ \\
     &  &  &  &  \\
     & \circ \ar[rr, "\rm{id} \times F_{u,v,v'',w''}"'] &  & \circ \ar[uur, "F_{v'',w'',w',z}\times \rm{id}"']&  
    \end{tikzcd}
    \]
    \end{subfigure}
    \caption{}
    \label{hex figure}
\end{figure}
\end{enumerate}
\noindent
\end{lemma}

The hexagon relation is a consistency check for compositions of $2$-morphisms along the faces of $3$-dimensional sub-cubes of $\underline{2}^n$, cf. \cite[Figure 24]{2001.00077}. 

Using Lemma \ref{lem:functor}, the following data gives the oriented Khovanov Burnside functor $F_{\rm or}\! : {\underline 2}^n \longrightarrow {\mathscr B}_{\sigma}$. $F(u)$ is the finite set $B(u)$, defined as the collection of generators $B(W(u),f_{\mathrm{can}})$ for the web space $\HOM_{\foam}(\emptyset, W(u))$ at the vertex $u$, constructed in Definition \ref{basis def}. For each edge $u\geq_1 v$ in $\underline{2}^n$, and $y \in F_{\rm or}(v)$, consider the value of the edge differential applied to $y$:
$\sum_{x\in F_{\rm{or}}(u) } \epsilon_{x,y} x$.
By Proposition \ref{prop:zipunzip}, each coefficient $\epsilon_{x,y}$ is an element of $\{ -1,0,1\}$.
Define
\[
F_{\rm or}(\phi_{u,v})=\{ (y,x) \in F_{\rm or}(v) \times F_{\rm or}(u) |
\epsilon_{x,y}=\pm 1 \},
\]
where the source and target maps are the
projections, and
the sign on elements of $F_{\rm or}(\phi_{u,v})$ is given by
$\epsilon_{x,y}$.

\subsection{Ladybug configurations and the completion of the construciton} \label{ladybug sec}
The Khovanov homotopy type of \cite{MR3230817, MR4153651} carries more information than  the Khovanov chain complex, and the additional ingredient in the construction is the analysis of {\em ladybug configurations} \cite[Section 5.4]{MR3230817}.
In the setting of the Khovanov homotopy type there is a unique choice of the $2$-morphisms $F_{u,v,v',w}$ for square faces in Lemma \ref{lem:functor}, except for the ladybug configuration case. 

As indicated in Figure \ref{ladybug}, such a configuration consists of a simple closed curve and two surgery arcs with endpoints linked on the circle, so that either surgery is a split into two circles, and the result of both surgeries is again a single component. In the Khovanov chain complex this corresponds to the generator $1$ being sent by both edge maps (comultiplications) to $1\otimes X+X\otimes 1$; then the multiplication map takes the result to $2X$. Defining the $2$-morphism $F_{u,v,v',w}$ in this case amounts to establishing a bijection between the summands $1\otimes X, X\otimes 1$ in the two $2$-component resolutions; this has to be done consistently so that the hexagon relation in Figure \ref{hex figure} holds for each $3$-dimensional cube. This is equivalent to specifying a bijection between the components of these two resolutions; one such correspondence -- the ``right pair'' in the terminology of \cite{MR3230817} -- is shown in Figure \ref{ladybug}. (The labels $a,b$ chosen for illustration here differ from those in \cite{MR3230817} to avoid confusion with the labels $1,2$ that are  used in the $\glnn{2}$ web context.) The analysis of $3$-dimensional cube configurations in \cite[Section 5.5]{MR3230817}, reformulated in the language of the Burnside category, is used to prove the hexagon relation in \cite{MR4153651} when a consistent (right or left) choice is made for all ladybug configurations.

\begin{figure}[ht] 
\centering
\includegraphics[height=2.5cm]{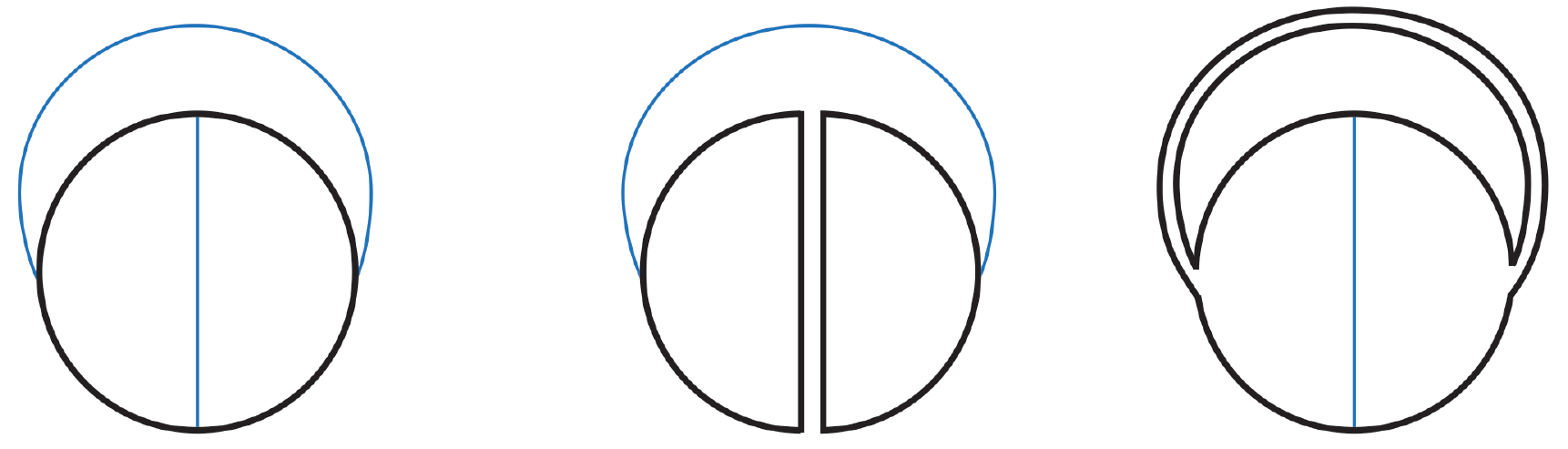}
{\small
\put(-152,10){$a$}
\put(-108,37){$b$}
\put(-65,10){$a$}
\put(-24,37){$b$}}
\caption{Ladybug configuration}
\label{ladybug}
\end{figure}
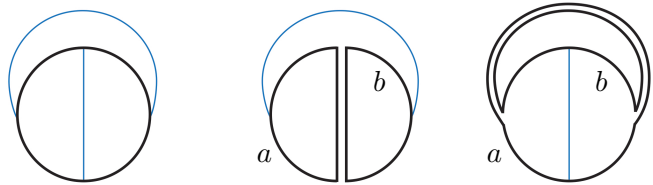

In the Blanchet setting
there is still a unique choice of $2$-morphisms $F_{u,v,v',w}$, except for the web version of ladybug configurations.
Such a configuration is a web $W$ with two zip/unzip moves so that the combinatorics of $c(W)$ (the $1$-labeled curves of $W$, in the terminology of Section \ref{Blanchet sec}) matches that in Figure \ref{ladybug}.
(Note that arbitrary $2$-labeled edges may be part of the configuration, and are not pictured.)

Proposition \ref{prop:zipunzip} gives the sign coherence of maps in a ladybug configuration. The value of the composition of the two corresponding edge maps in the cube of resolutions on the undotted generator $x$ is $\pm 2$ times the dotted generator $z$. (Here dotted and undotted generators are understood in the context of Definition \ref{basis def}.) The associated composition of signed correspondences has two elements in the set $s_A^{-1}(x)\cap t_A^{-1}(z)$ of the same sign, using the notation from (\ref{eq:signed correspondence}). As in the setting of the Khovanov homotopy type, a consistent choice (say, right pair) is made, ensuring that the hexagon in Lemma \ref{lem:functor} commutes. More specifically, the analysis in \cite[Proposition 6.1]{MR3611723} shows that the composition of the six $2$-morphisms corresponding to traversing the six faces of a $3$-dimensional cube induces the identity map on the $2$-element correspondence in the Burnside category. In our setting the correspondences have signs, which are determined by the signs of the Blanchet differential. However the combinatorics of the ladybug configurations is based on the $1$-labeled curves; moreover the square faces in the Blanchet complex commute, with no sign corrections. Thus the composition of the six $2$-morphisms is again the identity, verifying the hexagon relation.

With a functor to the signed Burnside category given by Lemma \ref{lem:functor}, the rest of the construction -- little boxes refinement (with reflections to accommodate signs) and spacial realization -- follows as in \cite{MR4078823}.

\begin{remark} \label{coefficients cancellations}
The coefficients of the edge maps with respect to the chosen generators are elements of $\{0, \pm 1\}$; in this respect our context is similar to that of \cite{MR4078823}, necessitating the use of signed Burnside categories. 
However, unlike the setting of the odd Khovanov homology in \cite{MR4078823}, the $2$-dimensional faces of the cube of resolutions in our context commute (with no sign correction required).

Note also that similarly to the odd Khovanov homology context, the differentials in the Blanchet theory have signs but no {\em cancellations} (where a generator is sent by an edge map to a linear combination of generators which then cancel when acted on by another edge map). 
The presence of cancellations leads to a substantially more involved structure of the moduli spaces in the framed flow category, cf. \cite[Section 4]{2011.11234}.
\end{remark}

\begin{remark}
It is pointed out in
\cite[Remark 4.22]{MR4153651}
that incorporating the Bar-Natan and Lee deformations in the construction of the Khovanov homotopy type presents a problem.
The results of Section \ref{sec: Bases for web spaces} hold for the deformed foam categories, see Remark \ref{deformations}, however  working in this setting does not appear to help in constructing a stable homotopy refinement of the deformed theories. The problem with the Lee deformation for $\glnn{2}$ foams is identical to that in \cite[Figure 4.1]{MR4153651}. For the Bar-Natan deformation, the complication is due to the presence of cancellations (discussed in the preceding remark). 
There is a bit of flexibility in choosing the location of dots in the generators of the Blanchet theory; moving them across $2$-labeled sheets results in a sign change, as shown in equation (\ref{eq:dotslide}). Nevertheless, it seems unlikely that there is a consistent choice of generators eliminating cancellations.
\end{remark}

\subsection{Reidemeister moves}
\label{sec:Rmoves}
In this section, we prove the homotopy invariance of ${\mathcal X}_{\rm or}(L)$ under Reidemeister moves.
The key facts we need to know about Blanchet's version of Khovanov homology, to lift its Reidemeister invariance properties to the level of the stable homotopy type ${\mathcal X}_{\rm or}$, are collected in the following proposition.

\begin{figure}[ht]
    \centering
\begin{tikzpicture}[scale=.3,rotate=90,anchorbase]
\draw [very thick] (0,0) to (0,1) to [out=90,in=0] (-.5,1.5) to [out=180,in=90] (-1,1);
\draw [white,line width=0.15cm] (-1,1) to [out=270,in=180] (-.5,.5) to [out=0,in=270] (0,1)to (0,2);
\draw [very thick]
(-1,1) to [out=270,in=180] (-.5,.5) to [out=0,in=270] (0,1)to (0,2.25);\end{tikzpicture}
$\leftrightarrow$
\begin{tikzpicture}[scale=.3,rotate=90,anchorbase]
\draw [very thick] (0,0) to (0,2.25);
\end{tikzpicture}
, \quad
\begin{tikzpicture}[scale=.3,rotate=90,anchorbase]
\draw [white,line width=0.15cm] (1,0) to [out=90,in=270](0,1.5)to [out=90,in=270]  (1,3);
\draw [very thick] (1,0) to [out=90,in=270](0,1.5)to [out=90,in=270]  (1,3);
\draw [white,line width=.15cm] (0,0) to [out=90,in=270](1,1.5)to [out=90,in=270]  (0,3);
\draw [very thick] (0,0) to [out=90,in=270](1,1.5)to [out=90,in=270]  (0,3);
\end{tikzpicture}
$\leftrightarrow$
\begin{tikzpicture}[scale=.3,rotate=90,anchorbase]
\draw [very thick] (0,0) to (0,3);
\draw [very thick] (1,0) to (1,3);
\end{tikzpicture}
, \quad
\begin{tikzpicture}[scale=.3,rotate=90,anchorbase]
\draw [white,line width=0.15cm] (2,0) to [out=90,in=270] (0,3) to [out=90,in=270] (2,6);
\draw [very thick] (2,0) to [out=90,in=270] (0,3)to [out=90,in=270] (2,6);
\draw [white,line width=0.15cm] (1,0) to [out=90,in=270](0,1.5) to [out=90,in=270] (2,4.5) to [out=90,in=270] (1,6);
\draw [very thick] (1,0) to [out=90,in=270](0,1.5)to [out=90,in=270]  (2,4.5) to [out=90,in=270] (1,6);
\draw [white,line width=0.15cm] (0,0) to [out=90,in=270] (2,3) to [out=90,in=270] (0,6);
\draw [very thick] (0,0) to [out=90,in=270] (2,3) to [out=90,in=270] (0,6);
\end{tikzpicture}
$\leftrightarrow$
\begin{tikzpicture}[scale=.3,rotate=90,anchorbase]
\draw [very thick] (2,0) to (2,6);
\draw [very thick] (1,0) to (1,6);
\draw [very thick] (0,0) to (0,6);
\end{tikzpicture}
    \caption{Reidemeister moves}
    \label{fig:Rmoves}
\end{figure}
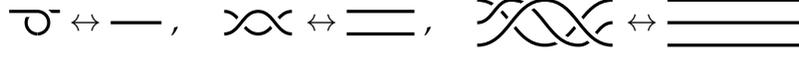

\begin{proposition} \label{Reid prop}
Let $X$ denote one of the Reidemeister moves from Figure~\ref{fig:Rmoves}, between two oriented link diagrams $D_1$ and $D_2$, where $D_1$ has more crossings than $D_2$. Let $r_X\colon \CKhor(D_1)\longrightarrow \CKhor(D_2)$ denote the associated chain map. Then $r_X$ is a deformation retract which factors as 
\[r_X = \phi \circ r_k\circ \cdots \circ  r_1,\] 
where $k\in \N$ and 
\begin{itemize}
    \item each $r_i$ for $1\leq i\leq k$ is a deformation retract induced by ``Gaussian elimination'' of an acyclic sub- or quotient            complex that is spanned by a pair of standard basis elements of $\CKhor(D_1)$ if $i=1$, respectively by the images of such a pair under $r_{i-1}\circ\cdots \circ r_1$ if $i>1$.
    \item there exists a bijection $\sigma$ between the images of standard basis elements $b$ of $\CKhor(D_1)$, which are not cancelled in this process, and the standard basis elements of $\CKhor(D_2)$, such that the foams $b$ and $\sigma(b)$ (Definition \ref{basis def}) are isotopic upon forgetting 2-labelled facets and overall minus signs.
     \item $\phi$ is an isomorphism of chain complexes, sending the image of standard basis elements $b$ in $r_k\circ\cdots \circ r_1( \CKhor(D_1))$ to $\pm \sigma(b)$ in $\CKhor(D_2)$.
\end{itemize}
\end{proposition}
\begin{proof}
This follows from the discussion of Reidemeister moves in \cite{MR3230817} (building on \cite{MR2174270}) combined with the fact that Blanchet foams reduce to Bar-Natan cobordisms under the specialisation of coeffiecients to $\mathbb{F}_2$, and the desired Gaussian eliminations can be carried out over $\Z$ whenever they are possible over $\mathbb{F}_2$. 

More specifically, to construct $r_1$, we are looking for a pair of generators that form an acyclic sub- or quotient complex. The differentials in the chain complexes $\CKhor$ send standard basis elements to linear combinations of standard basis elements with coefficients from the set $\{-1,0,1\}$. In particular, a component of the differential is invertible if and only if its specialisation to $\mathbb{F}_2$ coefficients is invertible. Under this specialisation, foams agree with Bar-Natan cobordisms, and thus the discussion from \cite{MR3230817} guarantees the existence of a suitable pair of generators, and thus the existence of $r_1$. (This pair of generators is local, it exists and has a uniform description whenever the Reidemeister move of type $X$ is applied within a link diagram.) 

Now we observe that Gaussian elimination just cancels this pair of generators. Because the pair formed an acyclic sub- or quotient complex, the first intermediate complex $C_1$, i.e. the co-domain of $r_1$, consists of all remaining generators with unchanged differentials between them. In particular, $C_1$ inherits (not all) standard basis elements and a differential with coefficients from the set $\{-1,0,1\}$. Now we iterate this procedure. 

The final result $r_k\circ \cdots \circ  r_1 ( \CKhor(D_1))$ is reached, when the total rank of the chain groups equals that of $\CKhor(D_2)$. In that case, comparison with Bar-Natan cobordisms via $\mathbb{F}_2$ coefficients establishes the existence of the bijection $\sigma$. The isomorphism $\phi$ is then determined by (a single relevant component of) the local model for the Reidemeister chain map $r_X$ as defined by Blanchet.
\end{proof}

The proof that Reidemeister moves induce homotopy equivalences of ${\mathcal X}_{\rm or}(L)$ is analogous to the proof of  \cite[Theorem 1]{MR3611723} and \cite[Section 5.3]{MR4078823}.
The deformation retraction $r_X$ in Proposition \ref{Reid prop} 
gives a natural transformation $\eta \colon {\underline{2}}^1\times {\underline{2}}^n \longrightarrow \mathscr{B}_{\sigma}$ with the sign map of $F(\varphi_{1,0}\times {\rm Id}_v)$ determined by the isomorphism $\phi$ in Proposition \ref{Reid prop}. As in the references above, the natural transformation $\eta$ is a stable equivalence of functors, giving rise to a homotopy equivalence of spectra associated to link diagrams related by Reidemeister moves.

\subsection{Maps induced by cobordisms} 
The remaining work concerns establishing naturality of the stable homotopy type ${\mathcal X}_{\rm or}$ with respect to link cobordisms. As mentioned in the introduction, 
\cite{MR3189434} defined maps on the Khovanov homotopy type induced by a link cobordism represented as a sequence of Reidemeister moves and Morse moves. The well-definedness of this induced map has not yet been verified, but it is conjectured to hold. The construction of the cobordism maps in \cite[Section 3.3]{MR3189434} is given in the context of framed flow categories; its analogue in the (signed) Burnside category framework is discussed in \cite[Section 5.5]{MR4078823}. The result in our setting reads as follows.

\begin{lemma} \label{cob lemma}
An oriented link cobordism $C$ from $L$ to $L'$, presented as a sequence of Reidemeister and Morse moves, induces a map of spectra ${\mathcal X}_{\rm or}(L')\longrightarrow {\mathcal X}_{\rm or}(L)$. The induced map on cohomology is the map ${\rm Kh}_{\rm or}(C)\! : \Khor(L)\longrightarrow \Khor(L')$ in Theorem \ref{Blanchet thm}.
\end{lemma}

\begin{proof}
Following the approach of \cite{MR4078823}, we associate to elementary cobordisms natural transformations of oriented Burnside functors. This is a lift to the signed Burnside category of the map of functors $F_{Kh}$ of \cite{MR4153651}.

First consider a saddle cobordism between two $n$-crossing link diagrams, $L\longrightarrow L'$.
For the Khovanov homotopy type \cite{MR3189434} the natural transformation 
${\underline{2}}^1\times {\underline{2}}^n \longrightarrow \mathscr{B}$ is given by the functor $F_{Kh}$ associated to the link diagram $L''$ with an extra crossing, corresponding to the saddle. In the Blanchet theory the edge differentials are zip/unzip foams (see Section \ref{Blanchet review sec}) rather than saddles. 

Consider the $(n+1)$-dimensional cubical diagram of webs and foams, corresponding to the cone on the saddle map. Two $n$-dimensional sub-cubes are given by the link diagrams $L, L'$, and the additional edge direction corresponds to the saddle cobordism.
While this is not a cube associated to a link diagram, it has properties analogous to one. 
Indeed, it follows from Proposition \ref{prop:zipunzip} that this cube does not have cancellations, cf. Remark \ref{coefficients cancellations}. Moreover, the definition of ladybug configurations in Section \ref{ladybug sec} is based on $1$-labelled curves of a web, so saddles and zip/unzip foams are treated equally in ladybug analysis.
Therefore the functor ${\underline{2}}^1\times {\underline{2}}^n \longrightarrow \mathscr{B}_{\sigma}$ associated to a saddle cobordism may be defined as in Sections \ref{signed Burnside sec}, \ref{ladybug sec}. 

The natural transformation assigned to a cup (birth) is defined by correspondences which label the new trivial circle with the undotted cup generator; the sign of this correspondence is $+1$. For the cap (death) the correspondences involve the trivial circle $C$ labeled by the dotted cup generator. In this case the sign of the correspondence is $\epsilon$, determined by the flow on $C$, see Lemma \ref{cup cap}.
\end{proof}

\subsection{Relation to the original construction of the Khovanov homotopy type} \label{sec: Bl Khov}
In this section we show that ${\mathcal X}_{\rm or}(L)$ is homotopy equivalent to ${\mathcal X}(L)$ of \cite{MR3230817}. We start by summarizing the relevant fact relating the two homology theories.

\begin{lemma} \label{lem: Bl and Khov}
Let $L$ be an oriented link diagram. There is a chain map \[\Phi\colon\CKhor(L)\longrightarrow \mathrm{CKh}(L)\] inducing an isomorphism on homology.  On each chain group $\Phi$ is a diagonal matrix with diagonal entries $\pm 1$ with respect to the basis in Definition \ref{basis def} for $\CKhor(L)$ and the canonical Khovanov basis for $\mathrm{CKh}(L)$.
\end{lemma}

The proof follows from \cite[Proposition 4.3 and Theorem 4.6]{2019arXiv190312194B}. A concrete description of $\Phi$ is also given by the results of our Section \ref{sec: Bases for web spaces}.

\begin{proposition} \label{prop: Bl and Khov spectrum}
Let $L$ be an oriented link diagram. Then ${\mathcal X}_{\rm or}(L)$ is stably homotopy equivalent to ${\mathcal X}(L)$.
\end{proposition}

\begin{proof}
The statement follows from the claim that the Burnside functors 
\[F, F_{\rm or}\colon {\underline{2}}^n\longrightarrow \mathscr{B}_{\sigma} \] are naturally isomorphic, where $F$ (thought of as a functor ${\underline{2}}^n\longrightarrow \mathscr{B}_{\sigma} $ where all signs are $+1$) was constructed in \cite{MR4153651} and
$F_{\rm or}$ is the result of Section \ref{signed Burnside sec}. Here a {\em natural isomorphism} is a natural transformation $F_{\rm or}\longrightarrow F$ such that for each vertex $u$ the morphism 
$F_{\rm or}(u)\longrightarrow F(u)$ is an isomorphism in $\mathscr{B}_{\sigma}$. In the case at hand the natural isomorphism $\eta$ arises from the bijection between the basis in Definition \ref{basis def} and the canonical Khovanov basis; denote this bijection $\psi$. More precisely, 
$\eta \colon {\underline{2}}^1\times {\underline{2}}^n \longrightarrow \mathscr{B}_{\sigma}$ is defined 
on objects by $\eta(0,u) = F(u), \eta(1,u)=F_{\rm or}(u)$ for $u\in \{0,1\}^n$. Define $\eta$ on the edges $e_u : (1,u) \to (0,u)$ to be the signed correspondence 
\[
\eta(e_u) = \left( F_{\rm or}(u) \xleftarrow{\rm id} F_{\rm or}(u) \xrightarrow{\psi_u} F(u) \right),
\]
with the sign given by Lemma \ref{lem: Bl and Khov}. The fact that this indeed gives rise to a natural isomorphism can be seen for example by following the detailed discussion in \cite[Section 3.6]{2001.00077}.
An isomorphism of Burnside functors induces a stable homotopy equivalence as stated in the proposition, cf. \cite[Lemma 4.17]{MR4078823}
\end{proof}

\renewcommand*{\bibfont}{\small}
\setlength{\bibitemsep}{0pt}
\raggedright
\printbibliography

\end{document}